\documentclass[11pt]{article}
\usepackage{graphicx}
\usepackage{amsthm,amssymb,tikz-cd, floatrow}
\usepackage{mathtools}
\usepackage[symbol]{footmisc}
\date{}
\newtheorem{theorem}{Theorem}[section]
\newtheorem{lemma}[theorem]{Lemma}
\newtheorem{proposition}[theorem]{Proposition}
\newtheorem{corolary}[theorem]{Corollary}
\newtheorem{definition}[theorem]{Definition}

\newtheorem{observation}[theorem]{Observation}
\newtheorem{rem}[theorem]{Remark}

\begin{document}
	\title{\textbf{On Ultra $k$-Secure Sets in Graphs}}
	\maketitle
	
	\begin{center}
		\author{{\large\bf $^1$K Karthik and $^{2}$Chandru Hegde\footnote{Corresponding Author}} \\ ~ \\ \small \it $^1$ Department of Mathematics, Government Fisrt Grade College,\\ \small Uppinangady-574241, Karnataka, INDIA. \\ \small \it Department of Mathematics, Mangalore University\\ \small Mangalagangothri, Mangalore 574199. INDIA\\ \small \textbf{email:} 
		$^1$karthik1111820@gmail.com, $^2$chandrugh@gmail.com} 
	\end{center}
	
	\begin{abstract}
	In a graph $G=(V,E)$, let $S$ be a non empty subset of $V$. For any $x\in S$, the vertex $x$ and its neighbors inside $S$ are defenders of $x$, whereas those lying outside $S$ are attackers on $x$. An attack on $S$ is an $|S|$-tuple of pairwise disjoint sets of attackers on vertices of $S$, whereas a defense is that of defenders. An attack on $S$ is defendable if $S$ has a defense, which provides at least as many defenders as attackers for each of its vertices. The set $S$ is a secure set if every attack on $S$ is defendable. Further, $S$ is an ultra secure set if $S$ has a defense which successfully defends $S$ against every attack. For an integer $k\geq 0$, a $k$-secure set $S$ is a secure set in which for any attack on $S$, there is a defense of $S$ with $k$ additional defenders for every vertex of $S$ having neighbors outside $S$. As a generalization of ultra security, in this paper  ultra $k$-secure sets and ultra $k$-security number of a graph are introduced. A characterization of ultra $k$-secure sets is obtained and using which ultra $k$-security numbers of complete multipartite graphs and grid-like graphs are computed.
	\end{abstract}

\section{Introduction}
Secure sets\cite{brigham2007security} were introduced by R.C. Brigham et al. in 2007. Let $G=(V,E)$ be a simple connected graph and $S=\{x_1,x_2,\ldots,x_m\}\subseteq V$. For any $x\in S$, the vertices of $N[x]-S$ are attackers of $x$ whereas the vertices in $N[x]\cap S$ are defenders of $x$ with respect to $S$. An attack on $S$ is an ordered $m$-tuple $\mathcal{A}=(A_1, A_2, \ldots, A_m)$ of pairwise disjoint sets $A_i\subseteq N[x_i]-S$, and is said to be defendable if there exists a defense $\mathcal{D}$, which is an ordered $m$-tuple $(D_1,D_2,\ldots, D_m)$ of pairwise disjoint sets $D_i\subseteq N[x_i]\cap S$, such that $|D_i|\geq |A_i|$ for each $i$, $1\leq i\leq m$. The set $S$ is a secure set if every attack on $S$ is defendable. Further if $S$ has a defense which defends it against every attack, then $S$ is an ultra secure set\cite{petrie2014f,petrie2012security}. The security number $s(G)$ is the minimum cardinality of a secure set in $G$. The minimum cardinality of an ultra secure set in $G$ is the ultra security number $s^u(G)$. More results on secure sets and bounds on security numbers are found in \cite{dutton2009graph,dutton2008bounds,globalsecuritygridgraphs}. Security numbers of various products of graphs are obtained in \cite{lexicographicproductsecurity,productgraphsecurity,kozawa2009security,strongproductsecurity}. The security of Sierpiński graphs is discussed in \cite{menon2023security}. As a generalization of secure sets, $k$-secure sets were introduced in \cite{ksecurity}.   Notably, a 0-secure set is simply a secure set as defined in \cite{brigham2007security}.

The concept of security in graphs has significant potential for various real-world applications. Secure sets can be used to design secure networks that are resilient to attacks or failures. Determining the security number assists in assessing the minimum level of redundancy or connectivity needed to maintain network stability. Additionally, ultra secure sets provide a theoretical framework for designing robust resource allocation and defense strategies in communication and infrastructure networks, where multiple simultaneous failures or attacks must be effectively countered. They also have potential applications in fault-tolerant network design and the protection of critical facilities in distributed systems.

In this paper, ultra $k$-secure sets are introduced by extending the idea of ultra secure sets to $k$-secure sets. The ultra $k$-security number of a graph is defined as the minimum cardinality of an ultra $k$-secure set in $G$, denoted by $s_k^u(G)$. A characterization of ultra $k$-secure sets is derived. Using the characterization, some bounds on ultra $k$-security number are obtained. Further, the ultra $k$-security number of complete multipartite graphs, Cartesian products of paths and cycles are obtained. 

For a nonempty set $S\subseteq V$, the sets $Bord(S)=\{x\in S: N[x]-S\ne\emptyset\}$, $Int(S)=S-Bord(S)$ and  $\partial S=N[S]-S$ are the border, interior and boundary of $S$ respectively. For $S\subset V$, $\langle S \rangle$ denotes the subgraph induced by $S$. The basic graph theory notions used in this paper may be found in \cite{graphtheorydbwest2001}.


\section{Ultra $k$-secure sets}
A different approach to secure sets is presented in \cite{kdistancesecurity}, where secure sets are generalized to $k$-distance secure sets by viewing attacks and defenses as functions with appropriate domains and co-domains, and by considering attackers and defenders even from distances greater than 1. A comprehensive demonstration of this approach and its equivalence with secure sets of \cite{brigham2007security} can be found in \cite{kdistancesecurity}. Throughout the discussion, we use the notions of attack and defense from \cite{kdistancesecurity}. 

Let $S$ be a nonempty proper subset of $V$. An attack $A$ on $S$ is a function $A:\partial S\to Bord(S)$ such that for any $x\in \partial S$, $x$ is adjacent to $A(x)$. A defense $D$ of $S$ is a function $D:N[Bord(S)]\cap S\to Bord(S)$ such that for any $y\in N[Bord(S)]\cap S$, $D(y)$ is equal to $y$ or adjacent to $y$.  For $0\leq k< \Delta(G)$, a set $S$ is a $k$-secure set\cite{ksecurity} if for any attack $A$ on $S$, there exists a defense $D$ such that $|D^{-1}(x)|-|A^{-1}(x)|\geq k$ for all $x\in Bord(S)$. 

As there is no attack on $V$, it is a $k$-secure set for all $k$. The minimum cardinality of a $k$-secure set in $G$ is the $k$-security number\cite{ksecurity}, denoted by $s_k(G)$. The following is a characterization of $k$-secure sets.
\begin{theorem}\label{char k secure}\cite{ksecurity}
	For $0\leq k<\Delta(G)$, a set $S\subseteq V$ is a $k$-secure set if and only if for any $X\subseteq Bord(S)$, $|N[X]\cap S|\geq |N[X]-S|+k|X|$.
\end{theorem}

The following remark follows immediately by the Theorem \ref{char k secure}.

\begin{rem}\label{degree of border vertex}
	For an integer $k\geq 1$, if $S$ is a $k$-secure set, then $deg(x)\geq k+1$ for all $x\in Bord(S)$.
\end{rem}

\begin{definition}
	For $0\leq k<\Delta(G)$, a defense $D$ of $S\subset V$ is an ultra $k$-defense if for every attack $A$ on $S$, $|D^{-1}(x)|-|A^{-1}(x)|\geq k$ for all $x\in Bord(S)$.
	A nonempty set $S\subseteq V$ is an ultra $k$-secure set if $S=V$ or $S$ has an ultra $k$-defense.
\end{definition}

By the definition, every ultra $k$-secure set is a $k$-secure set. Ultra 0-secure sets coincide with the ultra secure sets defined in \cite{petrie2014f,petrie2012security}. 

\subsection{Properties of ultra $k$-secure sets}
For any vertex $x$ in $Bord(S)$, there is an attack that maps every neighbor of $x$ outside of $S$ to $x$. This makes the next statement clear.
\begin{proposition}\label{ultra k defense proposition}
	Let $S$ be an ultra $k$-secure set and $D$ be an ultra $k$-defense of $S$. Then $|D^{-1}(x)|\geq |N[x]-S|+k$ for all $x\in Bord(S)$.
\end{proposition}

The following result follows by the fact that the vertices belonging to different components have disjoint sets of defenders.
\begin{proposition}\label{min is connected}
	For an integer $k$ with $0\leq k<\Delta(G)$, if $S\subset V$ is an ultra $k$-secure set and $\langle C\rangle$ is a component of $\langle S\rangle$, then $C$ is an ultra $k$-secure set.
\end{proposition}

%

We derive a characterization of ultra $k$-secure sets as follows.
\begin{theorem}\label{char ultra k secure}
	For $0\leq k<\Delta(G)$, a set $S\subseteq V$ is an ultra $k$-secure set if and only if $|N[X]\cap S|\geq k|X|+\underset{x\in X}{\sum}|N[x]-S|$  for all $X\subseteq Bord(S)$.
\end{theorem}

\noindent To prove Theorem \ref{char ultra k secure}, we construct a new graph $G(S)$ using $S$ and its boundary. In a graph $H$, for any $x\in V(H)$ and $X\subseteq V(H)$, the notations $N^{H}(x), N^{H}[x], N^{H}[X]$ denote the respective neighborhoods. However, we may not specify the superscript whenever the context is clear. 

Let $S\subset V(G)$ and $Bord(S)= \{b_1,b_2,\ldots, b_l\}$. The graph $G(S)$ is constructed as follows. 
\begin{enumerate}
	\item[Step-1:] For $i=1$, initialize $H \coloneq \langle S\rangle$ and $x\coloneq b_1$. 
	\item[Step-2:] Consider $\langle N^{\langle (V-S)\cup \{x\}\rangle}[x]\rangle$ and relabel the vertices other than $x$ as $y_{x1}, y_{x2},\ldots, y_{xt_x}$ where $t_x=|V\big(\langle N^{\langle (V-S)\cup \{x\}\rangle}(x)\rangle\big)|$, so that none of these vertices have same labeling as vertices of $G$. Let $X$ be the resulting graph.
	\item[Step-3:] Set $H'=H\cup X$. 
	\item[Step-4:] If $1\leq i< l$, then reset $i\coloneq i+1$,  $H\coloneq H'$, $x\coloneq b_i$ and go to Step-2. If $i=l$, $G(S)=H'$ and stop.
\end{enumerate}
A demonstration of constructing $G(S)$ is shown in Figures \ref{G(S)fig1}, \ref{G(S)fig2}, \ref{G(S)fig3} and \ref{G(S)fig4} by taking $S=\{b_1,b_2,b_3,s_1\}$.

\begin{figure}[H]
	\begin{floatrow}[2]
		\ffigbox{}{

			\tikzset{every picture/.style={line width=0.75pt}} 
			
			\begin{tikzpicture}[x=0.75pt,y=0.75pt,yscale=-1,xscale=1]
				
				\draw  [fill={rgb, 255:red, 0; green, 0; blue, 0 }  ,fill opacity=1 ] (141,85) .. controls (141,82.79) and (142.79,81) .. (145,81) .. controls (147.21,81) and (149,82.79) .. (149,85) .. controls (149,87.21) and (147.21,89) .. (145,89) .. controls (142.79,89) and (141,87.21) .. (141,85) -- cycle ;
				\draw  [fill={rgb, 255:red, 0; green, 0; blue, 0 }  ,fill opacity=1 ] (141,116) .. controls (141,113.79) and (142.79,112) .. (145,112) .. controls (147.21,112) and (149,113.79) .. (149,116) .. controls (149,118.21) and (147.21,120) .. (145,120) .. controls (142.79,120) and (141,118.21) .. (141,116) -- cycle ;
				\draw  [fill={rgb, 255:red, 0; green, 0; blue, 0 }  ,fill opacity=1 ] (95,100) .. controls (95,97.79) and (96.79,96) .. (99,96) .. controls (101.21,96) and (103,97.79) .. (103,100) .. controls (103,102.21) and (101.21,104) .. (99,104) .. controls (96.79,104) and (95,102.21) .. (95,100) -- cycle ;
				\draw  [fill={rgb, 255:red, 0; green, 0; blue, 0 }  ,fill opacity=1 ] (176,60) .. controls (176,57.79) and (177.79,56) .. (180,56) .. controls (182.21,56) and (184,57.79) .. (184,60) .. controls (184,62.21) and (182.21,64) .. (180,64) .. controls (177.79,64) and (176,62.21) .. (176,60) -- cycle ;
				\draw  [fill={rgb, 255:red, 0; green, 0; blue, 0 }  ,fill opacity=1 ] (187,100) .. controls (187,97.79) and (188.79,96) .. (191,96) .. controls (193.21,96) and (195,97.79) .. (195,100) .. controls (195,102.21) and (193.21,104) .. (191,104) .. controls (188.79,104) and (187,102.21) .. (187,100) -- cycle ;
				\draw  [fill={rgb, 255:red, 0; green, 0; blue, 0 }  ,fill opacity=1 ] (141,156) .. controls (141,153.79) and (142.79,152) .. (145,152) .. controls (147.21,152) and (149,153.79) .. (149,156) .. controls (149,158.21) and (147.21,160) .. (145,160) .. controls (142.79,160) and (141,158.21) .. (141,156) -- cycle ;
				\draw  [fill={rgb, 255:red, 0; green, 0; blue, 0 }  ,fill opacity=1 ] (161,186) .. controls (161,183.79) and (162.79,182) .. (165,182) .. controls (167.21,182) and (169,183.79) .. (169,186) .. controls (169,188.21) and (167.21,190) .. (165,190) .. controls (162.79,190) and (161,188.21) .. (161,186) -- cycle ;
				\draw  [fill={rgb, 255:red, 0; green, 0; blue, 0 }  ,fill opacity=1 ] (141,42) .. controls (141,39.79) and (142.79,38) .. (145,38) .. controls (147.21,38) and (149,39.79) .. (149,42) .. controls (149,44.21) and (147.21,46) .. (145,46) .. controls (142.79,46) and (141,44.21) .. (141,42) -- cycle ;
				\draw  [fill={rgb, 255:red, 0; green, 0; blue, 0 }  ,fill opacity=1 ] (176,141) .. controls (176,138.79) and (177.79,137) .. (180,137) .. controls (182.21,137) and (184,138.79) .. (184,141) .. controls (184,143.21) and (182.21,145) .. (180,145) .. controls (177.79,145) and (176,143.21) .. (176,141) -- cycle ;
				\draw  [fill={rgb, 255:red, 0; green, 0; blue, 0 }  ,fill opacity=1 ] (122,185) .. controls (122,182.79) and (123.79,181) .. (126,181) .. controls (128.21,181) and (130,182.79) .. (130,185) .. controls (130,187.21) and (128.21,189) .. (126,189) .. controls (123.79,189) and (122,187.21) .. (122,185) -- cycle ;
				\draw    (145,42) -- (145,156) ;
				\draw    (99,100) -- (145,85) ;
				\draw    (99,100) -- (145,116) ;
				\draw    (145,116) -- (191,100) ;
				\draw    (145,85) -- (191,100) ;
				\draw    (145,116) -- (180,141) ;
				\draw    (145,85) -- (180,60) ;
				\draw    (180,59) -- (191,100) ;
				\draw    (191,100) -- (180,141) ;
				\draw    (145,156) -- (165,186) ;
				\draw    (145,156) -- (126,185) ;
				\draw    (126,185) -- (165,186) ;
				\draw    (99,100) -- (145,156) ;
				\draw   (76,84.5) .. controls (76,49.43) and (96.15,21) .. (121,21) .. controls (145.85,21) and (166,49.43) .. (166,84.5) .. controls (166,119.57) and (145.85,148) .. (121,148) .. controls (96.15,148) and (76,119.57) .. (76,84.5) -- cycle ;
				\draw  [fill={rgb, 255:red, 0; green, 0; blue, 0 }  ,fill opacity=1 ] (191,26) .. controls (191,23.79) and (192.79,22) .. (195,22) .. controls (197.21,22) and (199,23.79) .. (199,26) .. controls (199,28.21) and (197.21,30) .. (195,30) .. controls (192.79,30) and (191,28.21) .. (191,26) -- cycle ;
				\draw    (180,60) -- (195,26) ;
				
				\draw (82,79) node [anchor=north west][inner sep=0.75pt]   [align=left] {$\displaystyle b_{1}$};
				\draw (128,113) node [anchor=north west][inner sep=0.75pt]   [align=left] {$\displaystyle b_{2}$};
				\draw (124,65) node [anchor=north west][inner sep=0.75pt]   [align=left] {$\displaystyle b_{3}$};
				\draw (124,29) node [anchor=north west][inner sep=0.75pt]   [align=left] {$\displaystyle s_{1}$};
				\draw (185,49) node [anchor=north west][inner sep=0.75pt]   [align=left] {$\displaystyle a_{1}$};
				\draw (196,91) node [anchor=north west][inner sep=0.75pt]   [align=left] {$\displaystyle a_{2}$};
				\draw (184,133) node [anchor=north west][inner sep=0.75pt]   [align=left] {$\displaystyle a_{3}$};
				\draw (150,150) node [anchor=north west][inner sep=0.75pt]   [align=left] {$\displaystyle a_{4}$};
				\draw (106,182) node [anchor=north west][inner sep=0.75pt]   [align=left] {$\displaystyle v_{1}$};
				\draw (170,179) node [anchor=north west][inner sep=0.75pt]   [align=left] {$\displaystyle v_{2}$};
				\draw (58,66) node [anchor=north west][inner sep=0.75pt]   [align=left] {{\large \textbf{\textit{S}}}};
				\draw (173,16) node [anchor=north west][inner sep=0.75pt]   [align=left] {$\displaystyle v_{3}$};

			\end{tikzpicture}\caption{Graph G}\label{G(S)fig1}}
		
		\ffigbox{}{

			\tikzset{every picture/.style={line width=0.75pt}} 
			
			\begin{tikzpicture}[x=0.8pt,y=0.8pt,yscale=-1,xscale=1]
				
				\draw  [fill={rgb, 255:red, 0; green, 0; blue, 0 }  ,fill opacity=1 ] (251,134) .. controls (251,131.79) and (252.79,130) .. (255,130) .. controls (257.21,130) and (259,131.79) .. (259,134) .. controls (259,136.21) and (257.21,138) .. (255,138) .. controls (252.79,138) and (251,136.21) .. (251,134) -- cycle ;
				\draw  [fill={rgb, 255:red, 0; green, 0; blue, 0 }  ,fill opacity=1 ] (251,165) .. controls (251,162.79) and (252.79,161) .. (255,161) .. controls (257.21,161) and (259,162.79) .. (259,165) .. controls (259,167.21) and (257.21,169) .. (255,169) .. controls (252.79,169) and (251,167.21) .. (251,165) -- cycle ;
				\draw  [fill={rgb, 255:red, 0; green, 0; blue, 0 }  ,fill opacity=1 ] (205,149) .. controls (205,146.79) and (206.79,145) .. (209,145) .. controls (211.21,145) and (213,146.79) .. (213,149) .. controls (213,151.21) and (211.21,153) .. (209,153) .. controls (206.79,153) and (205,151.21) .. (205,149) -- cycle ;
				\draw  [fill={rgb, 255:red, 0; green, 0; blue, 0 }  ,fill opacity=1 ] (206,204) .. controls (206,201.79) and (207.79,200) .. (210,200) .. controls (212.21,200) and (214,201.79) .. (214,204) .. controls (214,206.21) and (212.21,208) .. (210,208) .. controls (207.79,208) and (206,206.21) .. (206,204) -- cycle ;
				\draw  [fill={rgb, 255:red, 0; green, 0; blue, 0 }  ,fill opacity=1 ] (251,91) .. controls (251,88.79) and (252.79,87) .. (255,87) .. controls (257.21,87) and (259,88.79) .. (259,91) .. controls (259,93.21) and (257.21,95) .. (255,95) .. controls (252.79,95) and (251,93.21) .. (251,91) -- cycle ;
				\draw    (255,91) -- (255,165) ;
				\draw    (209,149) -- (255,134) ;
				\draw    (209,149) -- (255,165) ;
				\draw    (209,149) -- (210,204) ;
				
				\draw (189,139) node [anchor=north west][inner sep=0.75pt]   [align=left] {$\displaystyle b_{1}$};
				\draw (238,163) node [anchor=north west][inner sep=0.75pt]   [align=left] {$\displaystyle b_{2}$};
				\draw (236,116) node [anchor=north west][inner sep=0.75pt]   [align=left] {$\displaystyle b_{3}$};
				\draw (236,75) node [anchor=north west][inner sep=0.75pt]   [align=left] {$\displaystyle s_{1}$};
				\draw (202,210) node [anchor=north west][inner sep=0.75pt]   [align=left] {$\displaystyle y_{11}$};

			\end{tikzpicture}\caption{Graph $H'$ for $i=1$}\label{G(S)fig2}}
	\end{floatrow}
	
\end{figure}

\begin{figure}[H]
	\begin{floatrow}[2]
		\ffigbox{}{
			\tikzset{every picture/.style={line width=0.75pt}} 
			
			\begin{tikzpicture}[x=0.8pt,y=0.8pt,yscale=-1,xscale=1]
				
				\draw  [fill={rgb, 255:red, 0; green, 0; blue, 0 }  ,fill opacity=1 ] (356,141) .. controls (356,138.79) and (357.79,137) .. (360,137) .. controls (362.21,137) and (364,138.79) .. (364,141) .. controls (364,143.21) and (362.21,145) .. (360,145) .. controls (357.79,145) and (356,143.21) .. (356,141) -- cycle ;
				\draw  [fill={rgb, 255:red, 0; green, 0; blue, 0 }  ,fill opacity=1 ] (356,172) .. controls (356,169.79) and (357.79,168) .. (360,168) .. controls (362.21,168) and (364,169.79) .. (364,172) .. controls (364,174.21) and (362.21,176) .. (360,176) .. controls (357.79,176) and (356,174.21) .. (356,172) -- cycle ;
				\draw  [fill={rgb, 255:red, 0; green, 0; blue, 0 }  ,fill opacity=1 ] (310,156) .. controls (310,153.79) and (311.79,152) .. (314,152) .. controls (316.21,152) and (318,153.79) .. (318,156) .. controls (318,158.21) and (316.21,160) .. (314,160) .. controls (311.79,160) and (310,158.21) .. (310,156) -- cycle ;
				\draw  [fill={rgb, 255:red, 0; green, 0; blue, 0 }  ,fill opacity=1 ] (311,211) .. controls (311,208.79) and (312.79,207) .. (315,207) .. controls (317.21,207) and (319,208.79) .. (319,211) .. controls (319,213.21) and (317.21,215) .. (315,215) .. controls (312.79,215) and (311,213.21) .. (311,211) -- cycle ;
				\draw  [fill={rgb, 255:red, 0; green, 0; blue, 0 }  ,fill opacity=1 ] (356,98) .. controls (356,95.79) and (357.79,94) .. (360,94) .. controls (362.21,94) and (364,95.79) .. (364,98) .. controls (364,100.21) and (362.21,102) .. (360,102) .. controls (357.79,102) and (356,100.21) .. (356,98) -- cycle ;
				\draw    (360,98) -- (360,172) ;
				\draw    (314,156) -- (360,141) ;
				\draw    (314,156) -- (360,172) ;
				\draw    (314,156) -- (315,211) ;
				\draw  [fill={rgb, 255:red, 0; green, 0; blue, 0 }  ,fill opacity=1 ] (356,211) .. controls (356,208.79) and (357.79,207) .. (360,207) .. controls (362.21,207) and (364,208.79) .. (364,211) .. controls (364,213.21) and (362.21,215) .. (360,215) .. controls (357.79,215) and (356,213.21) .. (356,211) -- cycle ;
				\draw    (360,172) -- (360,211) ;
				\draw  [fill={rgb, 255:red, 0; green, 0; blue, 0 }  ,fill opacity=1 ] (404,165) .. controls (404,162.79) and (405.79,161) .. (408,161) .. controls (410.21,161) and (412,162.79) .. (412,165) .. controls (412,167.21) and (410.21,169) .. (408,169) .. controls (405.79,169) and (404,167.21) .. (404,165) -- cycle ;
				\draw    (360,172) -- (408,191) ;
				\draw    (360,172) -- (408,165) ;
				\draw    (408,165) -- (408,191) ;
				\draw  [fill={rgb, 255:red, 0; green, 0; blue, 0 }  ,fill opacity=1 ] (404,191) .. controls (404,188.79) and (405.79,187) .. (408,187) .. controls (410.21,187) and (412,188.79) .. (412,191) .. controls (412,193.21) and (410.21,195) .. (408,195) .. controls (405.79,195) and (404,193.21) .. (404,191) -- cycle ;
				
				\draw (293,146) node [anchor=north west][inner sep=0.75pt]   [align=left] {$\displaystyle b_{1}$};
				\draw (340,170) node [anchor=north west][inner sep=0.75pt]   [align=left] {$\displaystyle b_{2}$};
				\draw (340,124) node [anchor=north west][inner sep=0.75pt]   [align=left] {$\displaystyle b_{3}$};
				\draw (340,88) node [anchor=north west][inner sep=0.75pt]   [align=left] {$\displaystyle s_{1}$};
				\draw (307,218) node [anchor=north west][inner sep=0.75pt]   [align=left] {$\displaystyle y_{11}$};
				\draw (354,218) node [anchor=north west][inner sep=0.75pt]   [align=left] {$\displaystyle y_{21}$};
				\draw (401,197) node [anchor=north west][inner sep=0.75pt]   [align=left] {$\displaystyle y_{22}$};
				\draw (412,155) node [anchor=north west][inner sep=0.75pt]   [align=left] {$\displaystyle y_{23}$};

			\end{tikzpicture}\caption{Graph $H'$ for $i=2$}\label{G(S)fig3}}
		\ffigbox{}{

			\tikzset{every picture/.style={line width=0.75pt}} 
			
			\begin{tikzpicture}[x=0.8pt,y=0.8pt,yscale=-1,xscale=1]
				
				\draw  [fill={rgb, 255:red, 0; green, 0; blue, 0 }  ,fill opacity=1 ] (310,140) .. controls (310,137.79) and (311.79,136) .. (314,136) .. controls (316.21,136) and (318,137.79) .. (318,140) .. controls (318,142.21) and (316.21,144) .. (314,144) .. controls (311.79,144) and (310,142.21) .. (310,140) -- cycle ;
				\draw  [fill={rgb, 255:red, 0; green, 0; blue, 0 }  ,fill opacity=1 ] (310,171) .. controls (310,168.79) and (311.79,167) .. (314,167) .. controls (316.21,167) and (318,168.79) .. (318,171) .. controls (318,173.21) and (316.21,175) .. (314,175) .. controls (311.79,175) and (310,173.21) .. (310,171) -- cycle ;
				\draw  [fill={rgb, 255:red, 0; green, 0; blue, 0 }  ,fill opacity=1 ] (264,155) .. controls (264,152.79) and (265.79,151) .. (268,151) .. controls (270.21,151) and (272,152.79) .. (272,155) .. controls (272,157.21) and (270.21,159) .. (268,159) .. controls (265.79,159) and (264,157.21) .. (264,155) -- cycle ;
				\draw  [fill={rgb, 255:red, 0; green, 0; blue, 0 }  ,fill opacity=1 ] (265,210) .. controls (265,207.79) and (266.79,206) .. (269,206) .. controls (271.21,206) and (273,207.79) .. (273,210) .. controls (273,212.21) and (271.21,214) .. (269,214) .. controls (266.79,214) and (265,212.21) .. (265,210) -- cycle ;
				\draw  [fill={rgb, 255:red, 0; green, 0; blue, 0 }  ,fill opacity=1 ] (310,97) .. controls (310,94.79) and (311.79,93) .. (314,93) .. controls (316.21,93) and (318,94.79) .. (318,97) .. controls (318,99.21) and (316.21,101) .. (314,101) .. controls (311.79,101) and (310,99.21) .. (310,97) -- cycle ;
				\draw    (314,97) -- (314,171) ;
				\draw    (268,155) -- (314,140) ;
				\draw    (268,155) -- (314,171) ;
				\draw    (268,155) -- (269,210) ;
				\draw  [fill={rgb, 255:red, 0; green, 0; blue, 0 }  ,fill opacity=1 ] (310,210) .. controls (310,207.79) and (311.79,206) .. (314,206) .. controls (316.21,206) and (318,207.79) .. (318,210) .. controls (318,212.21) and (316.21,214) .. (314,214) .. controls (311.79,214) and (310,212.21) .. (310,210) -- cycle ;
				\draw    (314,171) -- (314,210) ;
				\draw  [fill={rgb, 255:red, 0; green, 0; blue, 0 }  ,fill opacity=1 ] (358,164) .. controls (358,161.79) and (359.79,160) .. (362,160) .. controls (364.21,160) and (366,161.79) .. (366,164) .. controls (366,166.21) and (364.21,168) .. (362,168) .. controls (359.79,168) and (358,166.21) .. (358,164) -- cycle ;
				\draw    (314,171) -- (362,190) ;
				\draw    (314,171) -- (362,164) ;
				\draw    (362,164) -- (362,190) ;
				\draw  [fill={rgb, 255:red, 0; green, 0; blue, 0 }  ,fill opacity=1 ] (358,190) .. controls (358,187.79) and (359.79,186) .. (362,186) .. controls (364.21,186) and (366,187.79) .. (366,190) .. controls (366,192.21) and (364.21,194) .. (362,194) .. controls (359.79,194) and (358,192.21) .. (358,190) -- cycle ;
				\draw  [fill={rgb, 255:red, 0; green, 0; blue, 0 }  ,fill opacity=1 ] (345,102) .. controls (345,99.79) and (346.79,98) .. (349,98) .. controls (351.21,98) and (353,99.79) .. (353,102) .. controls (353,104.21) and (351.21,106) .. (349,106) .. controls (346.79,106) and (345,104.21) .. (345,102) -- cycle ;
				\draw  [fill={rgb, 255:red, 0; green, 0; blue, 0 }  ,fill opacity=1 ] (360,129) .. controls (360,126.79) and (361.79,125) .. (364,125) .. controls (366.21,125) and (368,126.79) .. (368,129) .. controls (368,131.21) and (366.21,133) .. (364,133) .. controls (361.79,133) and (360,131.21) .. (360,129) -- cycle ;
				\draw    (314,140) -- (364,129) ;
				\draw    (314,140) -- (349,102) ;
				\draw    (364,129) -- (349,102) ;
				
				\draw (248,145) node [anchor=north west][inner sep=0.75pt]   [align=left] {$\displaystyle b_{1}$};
				\draw (295,169) node [anchor=north west][inner sep=0.75pt]   [align=left] {$\displaystyle b_{2}$};
				\draw (293,123) node [anchor=north west][inner sep=0.75pt]   [align=left] {$\displaystyle b_{3}$};
				\draw (294,85) node [anchor=north west][inner sep=0.75pt]   [align=left] {$\displaystyle s_{1}$};
				\draw (261,216) node [anchor=north west][inner sep=0.75pt]   [align=left] {$\displaystyle y_{11}$};
				\draw (308,216) node [anchor=north west][inner sep=0.75pt]   [align=left] {$\displaystyle y_{21}$};
				\draw (355,196) node [anchor=north west][inner sep=0.75pt]   [align=left] {$\displaystyle y_{22}$};
				\draw (368,154) node [anchor=north west][inner sep=0.75pt]   [align=left] {$\displaystyle y_{23}$};
				\draw (368,122) node [anchor=north west][inner sep=0.75pt]   [align=left] {$\displaystyle y_{31}$};
				\draw (354,86) node [anchor=north west][inner sep=0.75pt]   [align=left] {$\displaystyle y_{32}$};

			\end{tikzpicture}\caption{Graph $H'=G(S)$ for $i=3$}\label{G(S)fig4}}
	\end{floatrow}
\end{figure}

\begin{observation}\label{rem1}
	Let $G$ be a graph and $S\subset V(G)$. Then,
	\begin{enumerate}
		\item $\langle S\rangle$ is a subgraph of $G(S)$.
		\item For any distinct $x,y\in Bord(S)$, $(N^{G(S)}[x]-S)\cap (N^{G(S)}[y]-S)=\emptyset$.
		\item For any $x\in S$, $|N^{G(S)}[x]-S|=|N^G[x]-S|$. 
	\end{enumerate}
\end{observation}


%

\begin{proposition}
	For any graph $G$ and a set $S\subset V(G)$, there exists only one attack on $S$ in $G(S)$.
\end{proposition}

\begin{proof} Let $x\in Bord(S)$ and $z\in N^{G(S)}[x]-S$ be arbitrary. 
	By (2) of Observation \ref{rem1}, $N^{G(S)}(z)\cap Bord(S)=\{x\}$. Thus, for any attack $A$ on $S$, $A(z)=x$ and hence there exists only one attack on $S$ in $G(S)$.
\end{proof}

\begin{lemma}\label{lemma-1}
	A set $S\subset V$ is an ultra $k$-secure set in $G$ if and only if $S$ is a $k$-secure set in $G(S)$. 
\end{lemma}
\begin{proof}
	Suppose $S$ is an ultra $k$-secure set in $G$ with an ultra $k$-defense $D$. Then by Proposition \ref{ultra k defense proposition}, $|D^{-1}(x)|\geq |N^G[x]-S|+k$ for all $x\in Bord(S)$. Let $A_S$ be the unique attack on $S$ in $G(S)$. By (3) of Observation \ref{rem1}, for every $x\in Bord(S)$, $|A_S^{-1}(x)|=|N^{G(S)}[x]-S|=|N^{G}[x]-S|$. Now, $|A_S^{-1}(x)|+k=|N^G[x]-S|+k\leq |D^{-1}(x)|  \text{ for all } x\in Bord(S)$.
	Since $\langle S\rangle$ is a subgraph of $G(S)$, $D$ is a defense of $S$ in $G(S)$ also. Since $A_S$ is the unique attack on $S$ in $G(S)$, $S$ is a $k$-secure set in $G(S)$. 
	
	Conversely, suppose $S$ is a $k$-secure set in $G(S)$ and $A_S$ is the unique attack on $S$ in $G(S)$. Then there exists a defense $D$ of $S$ in $G(S)$ such that $|D^{-1}(x)|-|A_S^{-1}(x)|\geq k$ for all $x\in Bord(S)$. Let $A$ be any attack on $S$ in $G$. Then,
	$$
	|A^{-1}(x)|+k\leq |N^G[x]-S|+k=|A_S^{-1}(x)|+k \leq |D^{-1}(x)|$$  for all $x \in Bord(S)$.
	Further, $D$ is a defense of $S$ in $G$ as well. Therefore $S$ is an ultra $k$-secure set in $G$.
\end{proof}

\begin{lemma}\label{lemma-2}
	A set $S$ is a $k$-secure set in $G(S)$ if and only if  for every $X\subseteq Bord(S)$, $|N^G[X]\cap S|\geq k|X|+\underset{x\in X}{\sum}|N^G[x]-S|$.
\end{lemma}
\begin{proof}
	Note that $(N^{G(S)}[x]-S)\cap (N^{G(S)}[y]-S)=\emptyset$ for any distinct $x,y\in Bord(S)$. Thus for any $X\subseteq Bord(S)$, $|N^{G(S)}[X]-S|=\underset{x\in X}{\sum}|N^{G(S)}[x]-S|=\underset{x\in X}{\sum}|N^G[x]-S|$. Also $N^{G(S)}[X]\cap S=N^G[X]\cap S$. Then by Theorem \ref{char k secure}, Lemma \ref{lemma-2} holds.
\end{proof}

\noindent Now the proof of Theorem \ref{char ultra k secure} follows directly by Lemma \ref{lemma-1} and Lemma \ref{lemma-2}. \\
 
 \noindent
  The following results provide essential criteria for a set $S\subset V$ to be an ultra $k$-secure set.
\begin{theorem}\label{border}
	A necessary condition for a set $S\subset V$ to be an ultra $k$-secure set is  $|Bord(S)|\leq \frac{1}{k+1}|S|$.
\end{theorem}
\begin{proof}
	Let $S$ be an ultra $k$-secure set. Then there is a defense $D$ of $S$ such that for every $x\in Bord(S)$, $|D^{-1}(x)|-|N[x]-S|\geq k$, which implies $\sum\limits_{x\in Bord(S)}\big(|D^{-1}(x)|-|N[x]-S|\big)\geq k|Bord(S)|$.
	Note that $\sum\limits_{x\in Bord(S)}|D^{-1}(x)|=|N[Bord(S)]\cap S|\leq |S|$. Further, for each $x\in Bord(S)$, $|N[x]-S|\ge 1$, therefore, $\sum\limits_{x\in Bord(S)}|N[x]-S|\geq |Bord(S)|$. 
	By combining all these, $|S|-|Bord(S)|\geq  k|Bord(S)|$ and hence $|Bord(S)|\leq \frac{1}{k+1}|S|$.
\end{proof}
The following corollary is straight forward by the observation that $|S|=|Int(S)|+|Bord(S)|$. 
\begin{corolary}\label{int non empty}
	For $k\geq 1$, if $S$ is an ultra $k$-secure set, then $Int(S)\ne \emptyset$. Further $|Int(S)|\geq k|Bord(S)|$.
\end{corolary}


\begin{corolary}\label{degreeatleastn-k}
	For any graph $G$ of order $n$ and an integer $k$ with $1\leq k\leq \Delta(G)-1$, every vertex of degree at least $n-k$ belongs to the intersection of all ultra $k$-secure sets.
\end{corolary}

\begin{proof}
	Let $S$ be an ultra $k$-secure set. Suppose $u\in V-S$ with $deg(u)\geq n-k$. Since $u$ is non adjacent to at most $k-1$ vertices other than $u$ and $|S|> k-1$, it follows that  $|N[u]\cap S|\geq |S|-k+1$. Therefore, $|Bord(S)|\ge |S|-k+1$. Now, by Theorem \ref{border}, $|S|\geq (k+1)(|S|-k+1)$, which implies that $|S| < k$, a contradiction.
\end{proof}

\subsection{Properties of ultra $k$-security number}
By the definition of ultra $k$-secure sets, the following proposition is straightforward. 

\begin{proposition}
	For any graph $G$ of order $n$, $k+1\leq s_k^u(G)\leq n$. 
\end{proposition}

By Corollary \ref{degreeatleastn-k}, for any $k\ge 1$ and $n\ge 2$, $s_k^u(K_n)=n$. The next result shows that the lower bound given in the preceding proposition can be achieved. 

For a graph $G$, a cut vertex is a vertex of $G$, the removal of which results in a disconnected or trivial graph. A block of $G$ is a maximal connected subgraph of $G$ with no cut vertex. The block graph $B(G)$ of a graph $G$ is the graph with blocks of $G$ as its vertices and two blocks are adjacent whenever they have a common cut vertex. If a block of $G$ is complete, it is referred as a complete block.  

\begin{theorem} 
	For any graph $G$ of order $n$, $s_k^u(G)= k+1$ if and only if $k=n-1$ or $G$ has a block $B_1$ of $k+1$ vertices and a block $B_2$ of two vertices having a common cut vertex of degree $k+1$ such that $B_1$ is a pendant vertex in $B(G)$, which is adjacent to $B_2$.  
\end{theorem}

\begin{proof} If $k=n-1$, then the result is trivial. Let $s_k^u(G)= k+1<n$ and $S$ be a minimum ultra $k$-secure set. Then $Bord(S)\neq\emptyset$. By Theorem \ref{border}, $|Bord(S)|\leq \frac{|S|}{k+1}=1$ and hence $|Bord(S)|=1$. Therefore, $|Int(S)|=k$. Since $S$ is a $k$-secure set, by Theorem \ref{char k secure}, every interior vertex is adjacent to the border vertex of $S$ and $|\partial S|=1$.  Thus $\langle S\rangle$ is a block with $k+1$ vertices and $\langle Bord(S)\cup \partial S\rangle$ is a block with two vertices, satisfying the above conditions.  The converse is straight forward.
\end{proof}

The subsequent results provide sufficient conditions for identifying graphs having ultra $k$-security number equal to their order. The following corollaries are consequences of Corollary \ref{degreeatleastn-k}.
\begin{corolary}
	For any graph $G$ of order $n$ and integer $k\ge 1$, if $\delta(G)\geq n-k$, then $s^u_k(G)=n$.
\end{corolary}

\begin{corolary}\label{Theorem dn-1}
For any graph $G$ of order $n$, let $d_{n-1}$ denote the number of vertices of degree $n-1$ in $G$. Then for any $k\geq \frac{n-d_{n-1}}{d_{n-1}}$, $s_k^u(G)=n$.
\end{corolary}

\begin{proof}
	Let $S$ be a minimal ultra $k$-secure set in $G$. By Corollary \ref{degreeatleastn-k}, every vertex of degree $n-1$ belongs to $S$. Assume $S$ is a proper subset of $V$. Then there is some $v\in \partial(S)$.  Since every vertex of degree $n-1$ is adjacent to $v$, $|Bord(S)|\ge d_{n-1}$. Now by Theorem \ref{border}, $|S|\ge (k+1)|Bord(S)|\ge (k+1)d_{n-1} $. Now, using $k\geq \frac{n-d_{n-1}}{d_{n-1}}$, we get $|S|\ge n$, a contradiction. Therefore, if $k\geq \frac{n-d_{n-1}}{d_{n-1}}$ then $s_k^u(G)=n$. 
\end{proof}

\begin{corolary}
	Let $G$ be a graph of order $n$ and $d_{n-1}$ be the number of vertices of degree $n-1$ in $G$. If $d_{n-1}\ge \frac{n}{k+1}$, then $s_k^u(G)= n$ for all $k\ge 1$. 
\end{corolary}

\begin{proof}
	By Corollary \ref{Theorem dn-1},  it is clear that  $s_k^u(G)=n$ or $s^u_k(G)\geq (k+1)d_{n-1}$.  Using $d_{n-1}\ge \frac{n}{k+1}$, we get $s_k^u(G)= n$ for all $k\ge 1$. 
\end{proof}

The vertex connectivity $\kappa(G)$ of a graph $G$ is the minimum number of vertices to be removed from $G$ to get a disconnected or trivial graph. The next result gives a relation between $\kappa(G)$ and $s_k^u(G)$.
\begin{theorem}
	For any graph $G$ of order $n$ and integer $k\ge 1$, if $ \kappa(G)\geq \frac{n}{k+1}$, then $s^u_k(G)=n$.
\end{theorem}
\begin{proof}
	Let $S$ be a minimum ultra $k$-secure set. Assume $S$ is a proper subset of $V$. Then by Corollary \ref{int non empty}, $Int(S)\ne \emptyset$, therefore $\langle V-Bord(S)\rangle$ is disconnected. Thus, $\kappa\leq |Bord(S)|$. Now by Theorem \ref{border}, it follows that $\kappa\leq \frac{s^u_k(G)}{k+1}$,  which contradicts the hypothesis of the theorem. Thus $S=V$.
\end{proof}

\subsection{Variation of $s_k^u(G)$ on addition/removal of edges}\label{Variation of $s_k^u(G)$ on addition/removal of edges}

In this subsection, we explore how the addition or removal of edges impacts the values of $s_k^u(G)$. By the definition of ultra $k$-secure sets its clear that $s_2^u(P_3)=2$.  The graph $C_3$ can be obtained from $P_3$ by adding an edge between the end vertices, resulting in $s_2^u(C_3)=3$. imilarly, straightforward calculations show that   $s_2^u(P_4)=2$ and $s_2^u(C_4)=4$. Thus, although it seems that adding edges between non-adjacent pairs of vertices in $G$ leads to an increase in $s_k^u(G)$. However, this is not the case.

 Consider the graphs $G$ and $G'$ represented by the Figure \ref{figure addition of edge decreases s_u^k(G)}, where $G'$ is formed by adding an edge between two non-adjacent vertices of $G$. It can be observed that the sets $\{v_1, v_2, v_3, v_4\}$ and $\{v_5, v_6, v_7\}$ are the minimum ultra 2-secure sets in $G$ and $G'$ respectively. Therefore, $s_2^u(G)=4$ and $s_2^u(G')=3$.

 For any integers $n>k\geq 1$, consider the graph $G$ formed by adding a bridge between a end vertex of $P_{k+1}$ and any vertex of $K_n$. Let $p_1, p_2, \ldots, p_{k+1}$ be the vertices of $P_{k+1}$. Let $G'$ be the graph obtained by adding the edges to $G$ so that the induced subgraph  $\langle\{p_1, p_2, \ldots, p_{k+1}\}\rangle$ of $G'$ is isomorphic to $K_{k+1}$. Then $s_k^u(G)=n$ and $s_k^u(G')=k+1$, which leads to the next theorem.
 
\begin{theorem}\label{theorem addition of edge decreases s_u^k(G)}
For any integers $n> k\geq 1$, there exist graphs $G$ and $G'$ such that
\begin{enumerate}
	\item $G'$ is obtained by adding $\frac{k(k-1)}{2}$ edges to $G$.
	\item $s_k^u(G)=n$ and $s_k^u(G')=k+1$.
\end{enumerate}
\end{theorem}
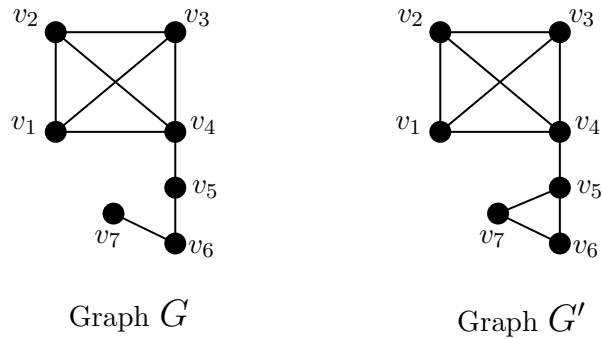
\begin{figure}[H]
\tikzset{every picture/.style={line width=0.75pt}} 

\begin{tikzpicture}[x=0.75pt,y=0.75pt,yscale=-1,xscale=1]
	
	\draw  [fill={rgb, 255:red, 0; green, 0; blue, 0 }  ,fill opacity=1 ] (140,115) .. controls (140,112.24) and (142.24,110) .. (145,110) .. controls (147.76,110) and (150,112.24) .. (150,115) .. controls (150,117.76) and (147.76,120) .. (145,120) .. controls (142.24,120) and (140,117.76) .. (140,115) -- cycle ;
	\draw  [fill={rgb, 255:red, 0; green, 0; blue, 0 }  ,fill opacity=1 ] (140,165) .. controls (140,162.24) and (142.24,160) .. (145,160) .. controls (147.76,160) and (150,162.24) .. (150,165) .. controls (150,167.76) and (147.76,170) .. (145,170) .. controls (142.24,170) and (140,167.76) .. (140,165) -- cycle ;
	\draw  [fill={rgb, 255:red, 0; green, 0; blue, 0 }  ,fill opacity=1 ] (200,115) .. controls (200,112.24) and (202.24,110) .. (205,110) .. controls (207.76,110) and (210,112.24) .. (210,115) .. controls (210,117.76) and (207.76,120) .. (205,120) .. controls (202.24,120) and (200,117.76) .. (200,115) -- cycle ;
	\draw  [fill={rgb, 255:red, 0; green, 0; blue, 0 }  ,fill opacity=1 ] (200,165) .. controls (200,162.24) and (202.24,160) .. (205,160) .. controls (207.76,160) and (210,162.24) .. (210,165) .. controls (210,167.76) and (207.76,170) .. (205,170) .. controls (202.24,170) and (200,167.76) .. (200,165) -- cycle ;
	\draw  [fill={rgb, 255:red, 0; green, 0; blue, 0 }  ,fill opacity=1 ] (200,221) .. controls (200,218.24) and (202.24,216) .. (205,216) .. controls (207.76,216) and (210,218.24) .. (210,221) .. controls (210,223.76) and (207.76,226) .. (205,226) .. controls (202.24,226) and (200,223.76) .. (200,221) -- cycle ;
	\draw  [fill={rgb, 255:red, 0; green, 0; blue, 0 }  ,fill opacity=1 ] (200,193) .. controls (200,190.24) and (202.24,188) .. (205,188) .. controls (207.76,188) and (210,190.24) .. (210,193) .. controls (210,195.76) and (207.76,198) .. (205,198) .. controls (202.24,198) and (200,195.76) .. (200,193) -- cycle ;
	\draw    (145,115) -- (205,115) ;
	\draw    (205,165) -- (145,115) ;
	\draw    (205,115) -- (145,165) ;
	\draw    (145,165) -- (205,165) ;
	\draw    (205,221) -- (205,115) ;
	\draw  [fill={rgb, 255:red, 0; green, 0; blue, 0 }  ,fill opacity=1 ] (169,206) .. controls (169,203.24) and (171.24,201) .. (174,201) .. controls (176.76,201) and (179,203.24) .. (179,206) .. controls (179,208.76) and (176.76,211) .. (174,211) .. controls (171.24,211) and (169,208.76) .. (169,206) -- cycle ;
	\draw    (145,115) -- (145,165) ;
	\draw    (205,221) -- (174,206) ;
	\draw  [fill={rgb, 255:red, 0; green, 0; blue, 0 }  ,fill opacity=1 ] (333,115) .. controls (333,112.24) and (335.24,110) .. (338,110) .. controls (340.76,110) and (343,112.24) .. (343,115) .. controls (343,117.76) and (340.76,120) .. (338,120) .. controls (335.24,120) and (333,117.76) .. (333,115) -- cycle ;
	\draw  [fill={rgb, 255:red, 0; green, 0; blue, 0 }  ,fill opacity=1 ] (333,165) .. controls (333,162.24) and (335.24,160) .. (338,160) .. controls (340.76,160) and (343,162.24) .. (343,165) .. controls (343,167.76) and (340.76,170) .. (338,170) .. controls (335.24,170) and (333,167.76) .. (333,165) -- cycle ;
	\draw  [fill={rgb, 255:red, 0; green, 0; blue, 0 }  ,fill opacity=1 ] (393,115) .. controls (393,112.24) and (395.24,110) .. (398,110) .. controls (400.76,110) and (403,112.24) .. (403,115) .. controls (403,117.76) and (400.76,120) .. (398,120) .. controls (395.24,120) and (393,117.76) .. (393,115) -- cycle ;
	\draw  [fill={rgb, 255:red, 0; green, 0; blue, 0 }  ,fill opacity=1 ] (393,165) .. controls (393,162.24) and (395.24,160) .. (398,160) .. controls (400.76,160) and (403,162.24) .. (403,165) .. controls (403,167.76) and (400.76,170) .. (398,170) .. controls (395.24,170) and (393,167.76) .. (393,165) -- cycle ;
	\draw  [fill={rgb, 255:red, 0; green, 0; blue, 0 }  ,fill opacity=1 ] (393,221) .. controls (393,218.24) and (395.24,216) .. (398,216) .. controls (400.76,216) and (403,218.24) .. (403,221) .. controls (403,223.76) and (400.76,226) .. (398,226) .. controls (395.24,226) and (393,223.76) .. (393,221) -- cycle ;
	\draw  [fill={rgb, 255:red, 0; green, 0; blue, 0 }  ,fill opacity=1 ] (393,193) .. controls (393,190.24) and (395.24,188) .. (398,188) .. controls (400.76,188) and (403,190.24) .. (403,193) .. controls (403,195.76) and (400.76,198) .. (398,198) .. controls (395.24,198) and (393,195.76) .. (393,193) -- cycle ;
	\draw    (338,115) -- (398,115) ;
	\draw    (398,165) -- (338,115) ;
	\draw    (398,115) -- (338,165) ;
	\draw    (338,165) -- (398,165) ;
	\draw    (398,221) -- (398,115) ;
	\draw  [fill={rgb, 255:red, 0; green, 0; blue, 0 }  ,fill opacity=1 ] (362,206) .. controls (362,203.24) and (364.24,201) .. (367,201) .. controls (369.76,201) and (372,203.24) .. (372,206) .. controls (372,208.76) and (369.76,211) .. (367,211) .. controls (364.24,211) and (362,208.76) .. (362,206) -- cycle ;
	\draw    (338,115) -- (338,165) ;
	\draw    (398,221) -- (367,206) ;
	\draw    (398,193) -- (367,206) ;
	
	\draw (122,102) node [anchor=north west][inner sep=0.75pt]   [align=left] {$\displaystyle v_{2}$};
	\draw (121,156) node [anchor=north west][inner sep=0.75pt]   [align=left] {$\displaystyle v_{1}$};
	\draw (208,102) node [anchor=north west][inner sep=0.75pt]   [align=left] {$\displaystyle v_{3}$};
	\draw (211,156) node [anchor=north west][inner sep=0.75pt]   [align=left] {$\displaystyle v_{4}$};
	\draw (212,188) node [anchor=north west][inner sep=0.75pt]   [align=left] {$\displaystyle v_{5}$};
	\draw (210,218) node [anchor=north west][inner sep=0.75pt]   [align=left] {$\displaystyle v_{6}$};
	\draw (163,213) node [anchor=north west][inner sep=0.75pt]   [align=left] {$\displaystyle v_{7}$};
	\draw (315,102) node [anchor=north west][inner sep=0.75pt]   [align=left] {$\displaystyle v_{2}$};
	\draw (314,156) node [anchor=north west][inner sep=0.75pt]   [align=left] {$\displaystyle v_{1}$};
	\draw (401,102) node [anchor=north west][inner sep=0.75pt]   [align=left] {$\displaystyle v_{3}$};
	\draw (404,156) node [anchor=north west][inner sep=0.75pt]   [align=left] {$\displaystyle v_{4}$};
	\draw (405,188) node [anchor=north west][inner sep=0.75pt]   [align=left] {$\displaystyle v_{5}$};
	\draw (403,218) node [anchor=north west][inner sep=0.75pt]   [align=left] {$\displaystyle v_{6}$};
	\draw (356,213) node [anchor=north west][inner sep=0.75pt]   [align=left] {$\displaystyle v_{7}$};
	\draw (122,149) node [anchor=north west][inner sep=0.75pt]   [align=left] {$ $};
	\draw (150,248) node [anchor=north west][inner sep=0.75pt]   [align=left] {{Graph \Large $G$}};
	\draw (345,249) node [anchor=north west][inner sep=0.75pt]   [align=left] {{Graph \Large $G'$}};

\end{tikzpicture}\caption{Example of a graph where the addition of an edge decreases $s_u^k(G)$.}\label{figure addition of edge decreases s_u^k(G)}
\end{figure}

For integers $n$ and $k$ with $n>2k\geq 2$, let $G$ be the graph formed by adding a bridge between a vertex of $K_n$ and a vertex of $K_{k+1}$. Let $v_n$ be the vertex of  $K_n$ incident with the bridge. Now we obtain a new graph $G'$ by adding edges between $v_n$ and all the vertices of $K_{k+1}$ that are non-adjacent to $v_n$ in $G$. Then, $s_k^u(G)=k+1$ and $s_k^u(G')=n$, hence the following result. 

\begin{theorem}\label{theorem addition of edge increases s_u^k(G)}
	For any integers $n>2k\geq 2$, there exist graphs $G$ and $G'$ such that 
	\begin{enumerate}
		\item $G'$ is obtained by adding $k$ edges to $G$.
		\item $s_k^u(G)=k+1$ and $k^u(G')=n$.
	\end{enumerate}
\end{theorem}

\noindent By Theorem \ref{theorem addition of edge decreases s_u^k(G)} and Theorem \ref{theorem addition of edge increases s_u^k(G)}, one may conclude that the addition/removal may increase/decrease the values of $s_k^u(G)$.

\section{Graphs with given ultra $k$-security number}
 In \cite{petrie2014f,petrie2012security}, it is shown that for each positive integer $s$, there exists a graph $G$ of order $n> 2s$  such that $s(G)=s^u(G)=s$. In this section, for given positive integers $n, k, s$ with $n\ge s \ge k+1\geq 2$, the construction of graphs of order $n$ having given ultra $k$-security number $s$ is demonstrated. If $s=n$ then $G=K_n$ is the desired graph for all $k\ge 1$. 
\begin{theorem}\label{graphconstructionforultraksecureset}
	For any positive integers $n,k,s$ with $n> s\ge k+1\ge 2$, there exists a graph $G$ of order $n$ with $s_k(G)=s_k^u(G)=s$.
\end{theorem}
\begin{proof} 
	Consider the following cases.
	\begin{enumerate}
		\item [\textbf{Case-1:}] $k\geq 2$.\\ 
	Let $q,r$ be the quotient and reminder obtained when $n-s$ is divided by $k$. For $1\leq i\leq q$, let $V_i=\{v_{ij}:1\leq j\leq k\}$. Let $U=\{u_i:1\leq i\leq s\}$. Let $V=U\cup\underset{i=1}{\overset{q}{\cup}}V_i$. Consider the following subcases.\\
	\textbf{Subcase-1: } $r=0$. 
	\\Consider the graph $G$ having vertex set $V$ and edge set $E$, determined by the following. 
	\begin{enumerate}
		\item[(i)] $\langle U\rangle$ is a block of $G$ isomorphic to $K_s$ and for each $i$ with $1\leq i\leq q$, $\langle V_i\rangle$ are blocks of $G$ isomorphic to $K_k$. 
		\item[(ii)] The edges  $u_sv_{11}$ and $v_{ik}v_{(i+1)1}$, $1\leq i< q$ are the bridges of $G$.
	\end{enumerate}
	
	For the graph $G$ considered above, it is clear that $|U|=s\ge k+1$, $|\partial (U)|=1$ and $|Bord(U)|=1$. Hence by Theorem \ref{border}, $U$ is a minimal ultra $k$-secure set. Now we prove that $U$ is the only minimal ultra $k$-secure set in $G$. Let $S$ be any minimal ultra $k$-secure set in $G$ other than $U$.  Clearly $S$ is a proper subset of $V$ (otherwise $S$ contains an ultra $k$-secure set $U$, which contradicts the minimality of $S$). Suppose $V-U\not\subseteq S$. Then there exists an $x\in (V-U)\cap Bord(S)$. By Remark \ref{degree of border vertex}, it follows that $deg(x)\geq k+1$, which contradicts the fact that every vertex in $V-U$ is of degree at most $k$. Thus $V-U\subseteq S$. Then by Remark \ref{degree of border vertex}, $V-U$ is a subset of $Int(S)$ ($\because$ $deg(v)\leq k$ for all $v\in V-U$). Thus $Bord(S)\subseteq U$. Since $\langle U\rangle$ is complete, $Bord(S)=S\cap U$ and $\partial S=U-S$. Since $N[U]=U\cup \{v_{11}\}$, it follows that $|N[S\cap U]\cap S|=|S\cap U|+1$. Then by Theorem \ref{char ultra k secure}, 
	\begin{eqnarray}\label{eq-4}
		|S\cap U|+1 = |N[S\cap U]\cap S| &\geq& k|S\cap U|+|U-S||S\cap U| \nonumber\\  &=&|S\cap U|\left(k+|U-S|\right).
	\end{eqnarray}
	Since $k\ge 2$ and $|S\cap U|\ne 0$,  (\ref{eq-4}) holds only when $|U-S|=0$ or equivalently $\partial S=\emptyset$. This implies $U\subset S$, a contradiction to the minimality of $S$. Thus, $U$ is the only minimal ultra $k$-secure set in $G$. Similarly, by using Theorem \ref{char k secure}, $U$ is the only minimal $k$-secure set in $G$. Therefore $s_k(G)=s_k^u(G)=s$.\\
	 \textbf{Subcase-2: } $r\neq 0$.\\
	Let $W=\{w_i:1\leq i\leq r\}$. Consider the graph $H$ with vertex set $V\cup W$ such that $\langle V\rangle$ is the graph $G$ constructed in Subcase-1, $\langle W\rangle$ is a block of $H$ isomorphic to $K_r$ and $v_{lk}w_1$ is the bridge of $H$ in addition to the bridges of $G$. The graph construction is illustrated in Figure \ref{constructionultraksecure}. Then as in Case-1, $U$ is the only minimal ultra $k$ secure set in $H$, as well as it is the only minimum $k$-secure set in $H$. Hence $s_k(H)=s_k^u(H)=s$.
	\item [\textbf{Case-2: }]$k=1$.\\
	Let $n-s< s$. Consider the graph $G$ of order $n$ having exactly two complete blocks with order $s$ and $n-s+1$ containing a common cut vertex. Since $n-s+1\leq s$, the complete block of order $s$ forms a minimal 1-secure set/minimal ultra 1-secure set. 
	By similar argument as in Subcase-1 of Case-1, the vertex set of complete block with $s$ vertices is the only minimal 1-secure set/minimal ultra 1-secure set in $G$ and hence $s_1(G)=s_1^u(G)=s$.  
	
	Suppose $s\leq n-s$. Consider the graph $G$ of order $n$ having two disjoint complete blocks of order $s$ and $n-s$ with a bridge between them as illustrated in Figure \ref{graphsforcase2}. Similar to the previous argument, the vertex sets of two complete blocks of $G$ are the only minimal 1-secure sets/minimal ultra 1-secure sets in $G$. Since $s\leq n-s$, $s_1(G)=s_1^u(G)=s$. 
\end{enumerate}
\end{proof}

\begin{figure}[H]
	\centering

	\tikzset{every picture/.style={line width=0.75pt}} 
	
	\begin{tikzpicture}[x=0.75pt,y=0.75pt,yscale=-1,xscale=1]
		
		\draw   (32,119) .. controls (32,91.94) and (53.94,70) .. (81,70) .. controls (108.06,70) and (130,91.94) .. (130,119) .. controls (130,146.06) and (108.06,168) .. (81,168) .. controls (53.94,168) and (32,146.06) .. (32,119) -- cycle ;
		\draw   (164,120) .. controls (164,103.43) and (177.43,90) .. (194,90) .. controls (210.57,90) and (224,103.43) .. (224,120) .. controls (224,136.57) and (210.57,150) .. (194,150) .. controls (177.43,150) and (164,136.57) .. (164,120) -- cycle ;
		\draw   (251,120) .. controls (251,103.43) and (264.43,90) .. (281,90) .. controls (297.57,90) and (311,103.43) .. (311,120) .. controls (311,136.57) and (297.57,150) .. (281,150) .. controls (264.43,150) and (251,136.57) .. (251,120) -- cycle ;
		\draw   (393,120) .. controls (393,103.43) and (406.43,90) .. (423,90) .. controls (439.57,90) and (453,103.43) .. (453,120) .. controls (453,136.57) and (439.57,150) .. (423,150) .. controls (406.43,150) and (393,136.57) .. (393,120) -- cycle ;
		\draw   (485,120) .. controls (485,106.19) and (496.19,95) .. (510,95) .. controls (523.81,95) and (535,106.19) .. (535,120) .. controls (535,133.81) and (523.81,145) .. (510,145) .. controls (496.19,145) and (485,133.81) .. (485,120) -- cycle ;
		\draw  [fill={rgb, 255:red, 0; green, 0; blue, 0 }  ,fill opacity=1 ] (119,120.5) .. controls (119,119.12) and (120.12,118) .. (121.5,118) .. controls (122.88,118) and (124,119.12) .. (124,120.5) .. controls (124,121.88) and (122.88,123) .. (121.5,123) .. controls (120.12,123) and (119,121.88) .. (119,120.5) -- cycle ;
		\draw  [fill={rgb, 255:red, 0; green, 0; blue, 0 }  ,fill opacity=1 ] (169,120.5) .. controls (169,119.12) and (170.12,118) .. (171.5,118) .. controls (172.88,118) and (174,119.12) .. (174,120.5) .. controls (174,121.88) and (172.88,123) .. (171.5,123) .. controls (170.12,123) and (169,121.88) .. (169,120.5) -- cycle ;
		\draw  [fill={rgb, 255:red, 0; green, 0; blue, 0 }  ,fill opacity=1 ] (214,120.5) .. controls (214,119.12) and (215.12,118) .. (216.5,118) .. controls (217.88,118) and (219,119.12) .. (219,120.5) .. controls (219,121.88) and (217.88,123) .. (216.5,123) .. controls (215.12,123) and (214,121.88) .. (214,120.5) -- cycle ;
		\draw  [fill={rgb, 255:red, 0; green, 0; blue, 0 }  ,fill opacity=1 ] (258,120.5) .. controls (258,119.12) and (259.12,118) .. (260.5,118) .. controls (261.88,118) and (263,119.12) .. (263,120.5) .. controls (263,121.88) and (261.88,123) .. (260.5,123) .. controls (259.12,123) and (258,121.88) .. (258,120.5) -- cycle ;
		\draw  [fill={rgb, 255:red, 0; green, 0; blue, 0 }  ,fill opacity=1 ] (300,120.5) .. controls (300,119.12) and (301.12,118) .. (302.5,118) .. controls (303.88,118) and (305,119.12) .. (305,120.5) .. controls (305,121.88) and (303.88,123) .. (302.5,123) .. controls (301.12,123) and (300,121.88) .. (300,120.5) -- cycle ;
		\draw  [fill={rgb, 255:red, 0; green, 0; blue, 0 }  ,fill opacity=1 ] (398,120.5) .. controls (398,119.12) and (399.12,118) .. (400.5,118) .. controls (401.88,118) and (403,119.12) .. (403,120.5) .. controls (403,121.88) and (401.88,123) .. (400.5,123) .. controls (399.12,123) and (398,121.88) .. (398,120.5) -- cycle ;
		\draw  [fill={rgb, 255:red, 0; green, 0; blue, 0 }  ,fill opacity=1 ] (443,120.5) .. controls (443,119.12) and (444.12,118) .. (445.5,118) .. controls (446.88,118) and (448,119.12) .. (448,120.5) .. controls (448,121.88) and (446.88,123) .. (445.5,123) .. controls (444.12,123) and (443,121.88) .. (443,120.5) -- cycle ;
		\draw  [fill={rgb, 255:red, 0; green, 0; blue, 0 }  ,fill opacity=1 ] (490,120.5) .. controls (490,119.12) and (491.12,118) .. (492.5,118) .. controls (493.88,118) and (495,119.12) .. (495,120.5) .. controls (495,121.88) and (493.88,123) .. (492.5,123) .. controls (491.12,123) and (490,121.88) .. (490,120.5) -- cycle ;
		\draw    (121.5,120.5) -- (171.5,120.5) ;
		\draw    (216.5,120.5) -- (260.5,120.5) ;
		\draw    (302.5,120.5) -- (324,121) ;
		\draw    (381,121) -- (400.5,120.5) ;
		\draw    (445.5,120.5) -- (492.5,120.5) ;
		\draw    (171,168) .. controls (160,205) and (308,159) .. (305,195) ;
		\draw    (305,193) .. controls (307,162) and (451,204) .. (447,168) ;
		\draw  [fill={rgb, 255:red, 0; green, 0; blue, 0 }  ,fill opacity=1 ] (332,120.5) .. controls (332,121.33) and (332.67,122) .. (333.5,122) .. controls (334.33,122) and (335,121.33) .. (335,120.5) .. controls (335,119.67) and (334.33,119) .. (333.5,119) .. controls (332.67,119) and (332,119.67) .. (332,120.5) -- cycle ;
		\draw  [fill={rgb, 255:red, 0; green, 0; blue, 0 }  ,fill opacity=1 ] (349,120.5) .. controls (349,121.33) and (349.67,122) .. (350.5,122) .. controls (351.33,122) and (352,121.33) .. (352,120.5) .. controls (352,119.67) and (351.33,119) .. (350.5,119) .. controls (349.67,119) and (349,119.67) .. (349,120.5) -- cycle ;
		\draw  [fill={rgb, 255:red, 0; green, 0; blue, 0 }  ,fill opacity=1 ] (364,120.5) .. controls (364,121.33) and (364.67,122) .. (365.5,122) .. controls (366.33,122) and (367,121.33) .. (367,120.5) .. controls (367,119.67) and (366.33,119) .. (365.5,119) .. controls (364.67,119) and (364,119.67) .. (364,120.5) -- cycle ;
		
		\draw (71,49) node [anchor=north west][inner sep=0.75pt]   [align=left] {$\displaystyle K_{s}$};
		\draw (186,70) node [anchor=north west][inner sep=0.75pt]   [align=left] {$\displaystyle K_{k}$};
		\draw (273,69) node [anchor=north west][inner sep=0.75pt]   [align=left] {$\displaystyle K_{k}$};
		\draw (414,68) node [anchor=north west][inner sep=0.75pt]   [align=left] {$\displaystyle K_{k}$};
		\draw (495,69) node [anchor=north west][inner sep=0.75pt]   [align=left] {$\displaystyle K_r$};
		\draw (106,106) node [anchor=north west][inner sep=0.75pt]   [align=left] {$\displaystyle u_{s}$};
		\draw (169,106) node [anchor=north west][inner sep=0.75pt]   [align=left] {$\displaystyle v_{11}$};
		\draw (197,122) node [anchor=north west][inner sep=0.75pt]   [align=left] {$\displaystyle v_{1k}$};
		\draw (400,106) node [anchor=north west][inner sep=0.75pt]   [align=left] {$\displaystyle v_{l1}$};
		\draw (427,122) node [anchor=north west][inner sep=0.75pt]   [align=left] {$\displaystyle v_{lk}$};
		\draw (258,106) node [anchor=north west][inner sep=0.75pt]   [align=left] {$\displaystyle v_{21}$};
		\draw (282,122) node [anchor=north west][inner sep=0.75pt]   [align=left] {$\displaystyle v_{2k}$};
		\draw (492,106) node [anchor=north west][inner sep=0.75pt]   [align=left] {$\displaystyle w_{1}$};
		\draw (288,197) node [anchor=north west][inner sep=0.75pt]   [align=left] {$q$ \ times};

	\end{tikzpicture}\caption{Graph construction for Case-1 of Theorem \ref{graphconstructionforultraksecureset}.}\label{constructionultraksecure}
\end{figure}
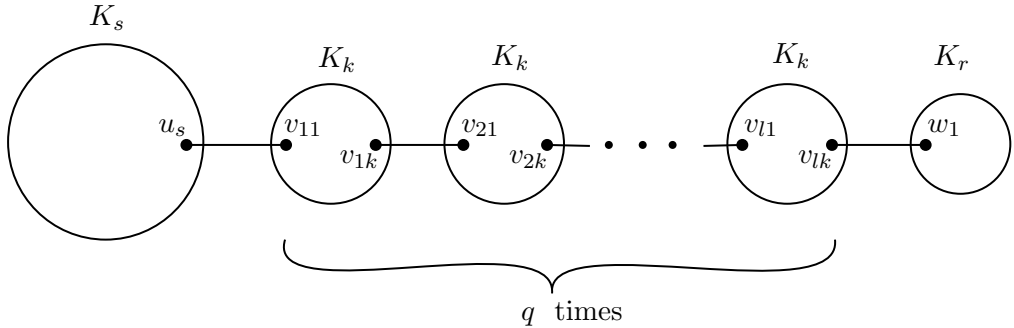

\begin{figure}[H]

	\tikzset{every picture/.style={line width=0.75pt}} 
	
	\begin{tikzpicture}[x=0.6pt,y=0.6pt,yscale=-1,xscale=1]
		
		\draw   (73,119) .. controls (73,87.52) and (98.52,62) .. (130,62) .. controls (161.48,62) and (187,87.52) .. (187,119) .. controls (187,150.48) and (161.48,176) .. (130,176) .. controls (98.52,176) and (73,150.48) .. (73,119) -- cycle ;
		\draw   (163,119) .. controls (163,87.52) and (188.52,62) .. (220,62) .. controls (251.48,62) and (277,87.52) .. (277,119) .. controls (277,150.48) and (251.48,176) .. (220,176) .. controls (188.52,176) and (163,150.48) .. (163,119) -- cycle ;
		\draw  [fill={rgb, 255:red, 0; green, 0; blue, 0 }  ,fill opacity=1 ] (170.5,121.75) .. controls (170.5,119.4) and (172.4,117.5) .. (174.75,117.5) .. controls (177.1,117.5) and (179,119.4) .. (179,121.75) .. controls (179,124.1) and (177.1,126) .. (174.75,126) .. controls (172.4,126) and (170.5,124.1) .. (170.5,121.75) -- cycle ;
		\draw   (359,121) .. controls (359,89.52) and (384.52,64) .. (416,64) .. controls (447.48,64) and (473,89.52) .. (473,121) .. controls (473,152.48) and (447.48,178) .. (416,178) .. controls (384.52,178) and (359,152.48) .. (359,121) -- cycle ;
		\draw   (520,122) .. controls (520,90.52) and (545.52,65) .. (577,65) .. controls (608.48,65) and (634,90.52) .. (634,122) .. controls (634,153.48) and (608.48,179) .. (577,179) .. controls (545.52,179) and (520,153.48) .. (520,122) -- cycle ;
		\draw    (459.5,123) -- (530.5,123) ;
		\draw  [fill={rgb, 255:red, 0; green, 0; blue, 0 }  ,fill opacity=1 ] (455.25,123) .. controls (455.25,120.65) and (457.15,118.75) .. (459.5,118.75) .. controls (461.85,118.75) and (463.75,120.65) .. (463.75,123) .. controls (463.75,125.35) and (461.85,127.25) .. (459.5,127.25) .. controls (457.15,127.25) and (455.25,125.35) .. (455.25,123) -- cycle ;
		\draw  [fill={rgb, 255:red, 0; green, 0; blue, 0 }  ,fill opacity=1 ] (526.25,123) .. controls (526.25,120.65) and (528.15,118.75) .. (530.5,118.75) .. controls (532.85,118.75) and (534.75,120.65) .. (534.75,123) .. controls (534.75,125.35) and (532.85,127.25) .. (530.5,127.25) .. controls (528.15,127.25) and (526.25,125.35) .. (526.25,123) -- cycle ;
		
		\draw (122,131) node [anchor=north west][inner sep=0.75pt]   [align=left] {$\displaystyle K_{s}$};
		\draw (204,129) node [anchor=north west][inner sep=0.75pt]   [align=left] {$\displaystyle K_{n-s+1}$};
		\draw (248,227) node [anchor=north west][inner sep=0.75pt]   [align=left] {};
		\draw (406,122) node [anchor=north west][inner sep=0.75pt]   [align=left] {$\displaystyle K_{s}$};
		\draw (554,127) node [anchor=north west][inner sep=0.75pt]   [align=left] {$\displaystyle K_{n-s+1}$};

	\end{tikzpicture}\caption{Graph construction for Case-2 of Theorem \ref{graphconstructionforultraksecureset}}\label{graphsforcase2}
\end{figure}

\section{Ultra $k$-security numbers of some classes of graphs}
In this section, the exact values/bounds on ultra $k$-security numbers are determined for paths, cycles, complete multipartite graphs and grid like graphs. 

\begin{proposition}
	\begin{enumerate}
		\item For any $n\geq 2$, $s^u_1(P_n)=2$.
		\item For any $n\geq 3$, $s^u_1(C_n)=\begin{cases}
			3& \text{if }n=3\\
			4& \text{if }n\geq 4.
		\end{cases}$
	\end{enumerate}
\end{proposition}
	 For $l\geq 2$, a graph $G=(V,E)$ is a complete $l$-partite graph, denoted by $K_{m_1,m_2,\ldots, m_l}$ if its vertex set can be partitioned into $V_1, V_2,\ldots, V_l$ of cardinalities $m_1,m_2,\ldots, m_l$ respectively such that for every $i$, none of the vertices in $V_i$ are adjacent to each other and every vertex in $V_i$ is adjacent to all vertices of $V_j$ for $j\neq i$.

\begin{theorem}
	For $l\geq 2$, let $m_1\leq m_2\leq\ldots\leq m_l$ be positive integers with $m_l>1$ and $\sum\limits_{i=1}^{l}m_i=n$. Then, $$s_k^u(K_{m_1,m_2,\ldots,m_l})=\begin{cases}
		n- \left\lfloor\dfrac{m_l-(k-1)(n-m_l)}{1+n-m_l}\right\rfloor& \text{  if  }1\leq k\leq \lfloor\frac{m_l-1}{n-m_l}\rfloor,m_l\geq \lceil\frac{n+1}{2}\rceil\\ \\
		\hspace{3cm}n&\text{  otherwise}.
	\end{cases}$$
\end{theorem}
\begin{proof}
	Let $V_1,V_2,\ldots, V_l$ be the partite sets of $K_{m_1,m_2,\ldots, m_l}$ with $|V_i|=m_i$ for each $i$, $1\leq i\leq l$. Let $n=m_1+m_2+\ldots+m_l$. Suppose $S\subset V$ is a minimum ultra $k$-secure set. If there exists $i$ with $1\leq i\leq l$ such that $S\cap V_i=\emptyset$, then $S=Bord(S)$, which is a contradiction. Further if there exist distinct $i,j$ with $1\leq i,j\leq l$ such that $S-V_i,S-V_j\neq\emptyset$, then $S=Bord(S)$, which is a contradiction. Since $S\subset V$, there exists a unique $t$ with $1\leq t\leq l$ such that $S-V_t,S\cap V_t\ne\emptyset$ and for $1\leq i\leq l$ with $i\neq t$, $V_i\subseteq S$. Then $Int(S)=S\cap V_t$, $Bord(S)=\underset{i\neq t}{\cup}V_i$. By Theorem \ref{char ultra k secure},
	\begin{eqnarray*}
		|S\cap V_t|+|\underset{i\neq t}{\cup}V_i|&=&|N[Bord(S)]\cap S|\\&\geq& k|Bord(S)|+\underset{z\in Bord(S)}{\sum}|N[z]-S|\\&=&k(n-m_t)+|V_t-S|(n-m_t).
	\end{eqnarray*}
	Then $|S\cap V_t|\geq (k-1)(n-m_t)+|V_t-S|(n-m_t)$ and hence $$|V_t|=|V_t\cap S|+|V_t-S|\geq (k-1)(n-m_t)+(1+n-m_t)|V_t-S|.$$ Thus $|V_t-S|\leq \left\lfloor\dfrac{m_t-(k-1)(n-m_t)}{1+n-m_t}\right\rfloor$. Hence $|S|\geq n- \left\lfloor\dfrac{m_t-(k-1)(n-m_t)}{1+n-m_t}\right\rfloor$. Since $S\subset V$, $\left\lfloor\dfrac{m_t-(k-1)(n-m_t)}{1+n-m_t}\right\rfloor\geq 1$ which is true if and only if $m_t-k(n-m_t)\geq 1$. Thus $k\leq \left\lfloor\frac{m_t-1}{n-m_t}\right\rfloor$. Further $\lfloor\frac{m_t-1}{n-m_t}\rfloor\geq 1$ if and only if $m_t\geq \lceil\frac{n+1}{2}\rceil$. Thus $m_t=max\{m_i:1\leq i\leq l\}=m_l$. 
	
	On the other hand, let $m_l\geq \lceil\frac{n+1}{2}\rceil$ and $1\leq k\leq \left\lfloor\frac{m_t-1}{n-m_t}\right\rfloor$. Then the set $S$ containing $V_1, V_2,\ldots, V_{l-1}$ and any $m_l- \left\lfloor\dfrac{m_l-(k-1)(n-m_l)}{1+n-m_l}\right\rfloor$ vertices of $V_l$ is an ultra $k$-secure set. This completes the proof.
	\end{proof}

	 The Cartesian product $G_1\square G_2$ of graphs $G_1$ and $G_2$ is a graph with vertex set $V(G_1)\times V(G_2)$ and edge set $\{(u_i,v_j)(u_l, v_k): (u_i=u_l \text{ and } v_jv_k\in E(G_2)) \text{ or } (v_j=v_k \text{ and } u_iu_l\in E(G_1))\}$. In $G_1\square G_2$, the sets $Row(i)=\{(u_i, v_j): 1\leq j\leq |V(G_2)|\}$, $Col(j)=\{(u_i, v_j): 1\leq i\leq |V(G_1)|\}$ are called the $i^{th}$ row and $j^{th}$ column respectively. Note that each row in $G_1\square G_2$ induces $G_2$ whereas each column induces $G_1$. The Cartesian products of paths and cycles are called grid-like graphs. 
	 
	 The ultra security number of grid-like graphs are determined in \cite{petrie2012security}. Here we obtain the ultra $k$-security numbers of grid-like graphs for $1\leq k\leq 3$. Since $\Delta(G)=4$ for grid like graphs, $s_u^k(G)$ for these graphs are defined only for $k \le 3$. It is easy to observe that $s_1^u(P_2\square P_2)=4$ and $s_1^u(P_2\square P_n)=4$, for all $n\ge 3$. Throughout the discussion, $v_{ij}$ denotes the vertex corresponding to $Row(i)$ and $Col(j)$.
	
	\begin{theorem}\label{ultra 1 security pmpn}
		For integers $n\geq m\geq 3$, $s_1^u(P_m\square P_n)=2m$.
	\end{theorem}
	\begin{proof}
		Let $S$ be a minimal ultra 1-secure set. By Theorem \ref{min is connected}, $\langle S\rangle$ is connected. Since none of the rows/columns alone form a 1-secure set, $S$ intersects at least 2  columns and 2 rows. By the connectedness of $\langle S\rangle$, it follows that $S$ intersects 2 consecutive columns and 2 consecutive rows. Consider the following cases.
		\begin{itemize}
			\item[\bf Case 1:] $S\cap Col(j)\neq \emptyset$ for all $j$, $1\leq j\leq n$.\\ 
			If every column contains a border vertex of $S$, then $|Bord(S)|\geq n$. Then by Theorem \ref{border}, $|S|\geq 2n\geq 2m$. If $Col(j)\cap Bord(S)=\emptyset$ for some $j$ with $1\leq j\leq n$, then $Col(j)\subseteq Int(S)$ and hence $Col(j-1)\subseteq S$ or $Col(j+1)\subseteq S$. Then $|S|\geq 2m$.
			\item[\bf Case 2:] $S\cap Col(j)=\emptyset$ for some $j$, $1\leq j\leq n$. \\
			Since $\langle S\rangle$ is connected, there exists a smallest integer $r$ such that $Col(j)\cap S=\emptyset$ for all $j>r$ or there exists a largest integer $r$ such that $Col(j)\cap S=\emptyset$ for all $j<r$. Without loss of generality, assume that $r$ is the smallest integer such that $Col(j)\cap S=\emptyset$ for all $j>r$. Then $Col(r)\cap S\neq\emptyset$. Since none of the subset of $Col(1)$ is an ultra 1-secure set, $r\geq 2$. Now  $Col(r+1)\cap S=\emptyset$ implies $Col(r)\cap S\subseteq Bord(S)$.\\ 
			\noindent {\bf Claim:} $Col(r)\cap (V-S)=\emptyset$. \\	 	\noindent
			If $Col(r)\cap (V-S)\neq \emptyset$, then for $X=Col(r)\cap S$, $\underset{x\in X}{\sum}|N[x]-S|\geq |N[X]-S|\geq |X|+1$ and $|N[X]\cap S|\leq 2|X|$. By Theorem \ref{char ultra k secure}, $2|X|\geq |N[X]\cap S|\geq |X|+\underset{x\in X}{\sum}|N[x]-S|\geq 2|X|+1$, which is not possible. Thus claim holds.\\
			By the claim and using $Col(r+1)\cap S=\emptyset$, we get $Col(r)\subseteq Bord(S)$. Hence $|Bord(S)|\geq m$. Then by Theorem \ref{border}, $|S|\geq 2m$.
		\end{itemize}
		
		\noindent By the above cases, $s_1^u(P_m\square P_n)\geq 2m$. Equality holds by the observation that $S_1=Col(1)\cup Col(2)$ is an ultra 1-secure sets.
	\end{proof}
	
	By direct observation, $s_1^u(P_2\square C_3)=6$, $s_1^u(P_3\square C_3)=s_1^u(C_3\square C_3)=9$, $s_1^u(P_2\square C_m)=8$, $s_1^u(P_3\square C_n)=s_1^u(C_3\square C_n)=12$ for all $n\ge 4$. The following is similar to Theorem \ref{ultra 1 security pmpn} with some suitable modifications.
	\begin{theorem}
		For integers $m,n\geq 4$,  $s_1^u(P_m\square C_n)=min\{4m,2n\}$ and  $s_1^u(C_m\square C_n)=min\{4m,4n\}$.
	\end{theorem}

	For integers $m,n$ with $\min\{m,n\}\leq 4$, $s_2^u(P_m\square P_n)=mn$. Now we obtain bounds on ultra 2-security number of $P_m\square P_n$ for $m,n\geq 5$. 
	
	\begin{lemma}\label{ultra 2 secure pmpn lemma 1}
		For $m,n\geq 5$, let $S$ be an ultra 2-secure set of $P_m\square P_n$. If  $v_{ij}\in \partial S$ for some $i,j$ with $2<i<m-1$ and $2<j<n-1$, then $N[N[v_{ij}]]-\{v_{ij}\}\subseteq S$.
	\end{lemma}
	\begin{proof}
		Let $v_{ij}\in \partial S$  for some $i,j$ with $2<i<m-1$ and $2<j<n-1$. Then at least one of $v_{(i-1)j}, v_{(i+1)j}, v_{i(j-1)}, v_{i(j+1)}$ belongs to $Bord(S)$. Without loss of generality, assume that $v_{(i-1)j}\in Bord(S)$. Since $S$ is ultra 2-secure, $v_{(i-1)(j-1)}$, $v_{(i-1)(j+1)}$, $v_{(i-2)j}\in S$. \\
		\textbf{Claim: }$v_{(i-1)(j+1)}\notin Bord(S)$.\\
		Suppose $v_{(i-1)(j+1)}\in Bord(S)$. Then $|\partial S\cap\{v_{(i-2)(j+1)}, v_{i(j+1)}, v_{(i-1)(j+2)}\}|\geq 1$. 
		\begin{itemize}
			\item If $|\partial S\cap\{v_{(i-2)(j+1)}, v_{i(j+1)}, v_{(i-1)(j+2)}\}|> 1$, then the set $|N[v_{(i-1)(j+1)}]\cap S|\le 3$ and $2|\{v_{(i-1)(j+1)}\}|+ |N[v_{(i-1)(j+1)}]-S|>3$, which contradicts Theorem \ref{char ultra k secure}. Thus exactly one of $v_{(i-2)(j+1)}, v_{i(j+1)}, v_{(i-1)(j+2)}$ belongs to $\partial S$.
			\item Further, if $v_{(i-2)(j+1)}\in \partial S$ or $v_{(i-1)(j+2)}\in \partial S$, then $v_{i(j+1)}$ belongs to $S$ and hence to $Bord(S)$. Then  for $X=\{v_{(i-1)j}, v_{(i-1)(j+1)}, v_{i(j+1)}\}\subseteq Bord(S)$, $|N[X]\cap S|\le 8$ and $2|X|+\underset{x\in X}{\sum}|N[x]-S| \ge 9$, again a contradiction to  Theorem \ref{char ultra k secure}. Thus $v_{i(j+1)}\in \partial S$ and $v_{(i-2)(j+1)}, v_{(i-1)(j+2)}\in S$.
			\item Further,  $v_{i(j+2)}\in S$. (Otherwise, $X=\{v_{(i-1)j}, v_{(i-1)(j+1)}, v_{(i-1)(j+2)}\}\subseteq Bord(S)$ fails to satisfy Theorem \ref{char ultra k secure}). 
			\item Since $S$ is ultra 2-secure and $v_{i(j+1)}\in \partial S$, clearly $v_{(i+1)(j+2)}\in S$ and hence $v_{(i+1)(j+1)}\in S$ (otherwise, $X=\{v_{(i-1)(j+1)}, v_{i(j+2)}, v_{(i+1)(j+2)}\}\subseteq Bord(S)$ fails to satisfy Theorem \ref{char ultra k secure}).
		\end{itemize}
		By the above argument, $v_{i(j+1)}\in \partial S$ and hence $v_{(i+1)(j+1)}\in Bord(S)$. Now, $X=\{v_{(i-1)j}, v_{(i-1)(j+1)}, v_{(i+1)(j+1)}\}\subseteq Bord(S)$ fails to satisfy Theorem \ref{char ultra k secure}. Hence the claim holds. 
		
		By the claim, it follows that $v_{i(j+1)}\in S$. By symmetry, all the neighbors of $v_{ij}$ belong to $S$. Since $S$ is ultra 2-secure,  $N[N[v_{ij}]]-\{v_{ij}\}\subseteq S$. 
	\end{proof}
	\begin{observation}\label{obs every 3-m-2 row/ 3-n-2 column  intersects S}
		By Lemma \ref{ultra 2 secure pmpn lemma 1}, it follows that every $Row(i)\cap S\neq\emptyset$, $Col(j)\cap S\neq\emptyset$ for all $i,j$, $2< i<m-1$, $2< j<n-1$.
	\end{observation}
	\begin{lemma}\label{row-1-row-2}
		Let $S$ be an ultra 2-secure set of $P_m\square P_n$. Then $Row(i), Col(j) \subseteq S$ for all $i\in \{1, 2, m-1, m\}$ and $j\in \{1, 2, n-1, n\}$. 
	\end{lemma}
	\begin{proof}
		Let $S$ be an ultra 2-secure set in $P_m\square P_n$. Suppose $v_{ij}\in \partial S$ for some $i,j$, $1\leq i\leq 2$, $1\leq j\leq n$. Consider the following cases.
		\begin{enumerate}
			\item [\textbf{Case-1:}]$i=1$, $j=1$.\\
			Then without loss of generality, assume that $v_{12}\in S$. Since $S$ is ultra 2-secure, $v_{13}, v_{22}\in S$. If $v_{21}\notin S$, then for $X=\{v_{12}, v_{22}\}\subseteq Bord(S)$, $5\geq |N[X]\cap S|\geq 2|X|+\underset{z\in  X}{\sum}|N[z]-S|\ge 4+2=6$, a contradiction to Theorem \ref{char ultra k secure}. If $v_{21}\in S$, then $X=\{v_{12}, v_{21}\}$ leads to a contradiction.
			\item [\textbf{Case-2:}]$2\leq j\leq n-1$.\\
			If $|N(v_{ij})\cap S|>2$, then $X=N(v_{ij})\cap S\subseteq Bord(S)$ leads to a contradiction to Theorem \ref{char ultra k secure}. Thus $|N(v_{ij})\cap S|\leq 2$.
			Suppose $|N(v_{ij})\cap S|=1$. If $N(v_{ij})\cap S=\{v_{i(j-1)}\}$, then $X=\{v_{i(j-1)}, v_{(i+1)(j-1)}\}\subseteq Bord(S)$ leads to a contradiction to Theorem \ref{char ultra k secure} if $i=1$ and $X=\{v_{(i-1)(j-1)}, v_{i(j-1)}, v_{(i+1)(j-1)}\}\subseteq Bord(S)$ leads to a contradiction whenever $i=2$. If $N(v_{ij})\cap S=\{v_{i(j+1)}\}$, then a similar argument leads to a contradiction. If $N(v_{ij})\cap S=\{v_{(i+1)j}\}$, then $X=\{v_{(i+1)(j-1)}, v_{(i+1)(j+1)}, v_{(i+1)j}\}\subseteq Bord(S)$ leads to a contradiction. Further, if $N(v_{ij})\cap S=\{v_{(i-1)j}\}$, then $X=\{v_{(i-1)(j-1)}, v_{(i-1)(j+1)}, v_{(i-1)j}\}\subseteq Bord(S)$ leads to a contradiction.
			Therefore $|N(v_{ij})\cap S|\neq 1$. Suppose $|N(v_{ij})\cap S|= 2$. Then consider the following subcases.\\
			\textbf{Subcase-1: } $i=1$.\\
			If $N(v_{1j})\cap S=\{v_{1(j-1)}, v_{1(j+1)}\}$, then  as $S$ is ultra 2-secure, $v_{2(j+1)}, v_{2(j-1)}\in S$ and hence  $v_{2(j+1)}, v_{2(j-1)}\in Bord(S)$. Then $X=\{v_{1(j-1)}, v_{1(j+1)}, v_{2(j+1)}, v_{2(j-1)}\}$ $\subseteq Bord(S)$ leads to a contradiction to Theorem \ref{char ultra k secure}. If $N(v_{1j})\cap S=\{v_{1(j-1)}, v_{2j}\}$, then $X=\{v_{1(j-1)}, v_{2j}, v_{2(j+1)}\}$ leads to a contradiction to Theorem \ref{char ultra k secure}. If $N(v_{1j})\cap S=\{v_{1(j+1)}, v_{2j}\}$, then $X=\{v_{1(j-1)}, v_{2j}, v_{2(j-1)}\}$ leads to a contradiction.\\
			\textbf{Subcase-2: } $i=2$.\\
			Suppose $N(v_{2j})\cap S=\{u,w\}$. Since $S$ ultra 2-secure, all the neighbors of $u$ and $w$ except $v_{ij}$ belong to $S$. If $u$ and $w$ belong to same row/column, then it leads to a contradiction with the fact that $|N(v_{2j})\cap S|\leq 2$. Thus $u$ and $w$ must belong to different rows and columns. Without loss of generality, assume that $u=v_{3j}$ and $w=v_{2(j+1)}$. Since $S$ is ultra 2-secure, $v_{1(j+1)}, v_{3(j-1)}\in S$ and hence $v_{1(j+1)}, v_{3(j-1)}\in Bord(S)$. Then $X=\{v_{1(j+1)}, v_{3(j-1)}$, $v_{3j}, v_{2(j+1)}\}\subseteq Bord(S)$ leads to a contradiction to Theorem \ref{char ultra k secure}. 
		\end{enumerate}

		\noindent By the above cases, $|N(v_{ij})\cap S|=0$, which is a contradiction to the hypothesis that $v_{ij}\in \partial S$. Thus $v_{ij}\notin \partial S$ for any $i\in \{1, 2\}$ and $j\in \{1, 2, \ldots, n-1\}$. Now by symmetry, it follows that $v_{ij}\notin\partial S$ whenever $i\in \{1, 2, m-1, m\}$ or $j\in \{1, 2, n-1, n\}$. By Observation \ref{obs every 3-m-2 row/ 3-n-2 column  intersects S}, $Row(3), Row(m-2), Col(3), Col(n-2)$ intersect $S$. Thus it follows that $Row(i), Col(j) \subseteq S$ for all $i\in \{1, 2, m-1, m\}$ and $j\in \{1, 2, n-1, n\}$.
	\end{proof}

	
	\begin{lemma}\label{ultra 2 secure pmpn lemma 2}
		Let $m,n\geq 5$ and $S$ be any ultra 2-secure set in $P_m\square P_n$. Then the distance between any two vertices of $V-S$ is at least 5.
	\end{lemma}
	\begin{proof}  Let $x$, $y$ be any two vertices in $V-S$.
		As a consequence of Lemma \ref{ultra 2 secure pmpn lemma 1} and Lemma \ref{row-1-row-2}, it follows that $V-S=\partial S$ and $N(x)\cap N(y)=\emptyset$, $N(x)\cup N(y)\subseteq Bord(S)$ for any distinct $x, y\in \partial S$. Thus, distance between $x$ and $y$ is at least 3. Suppose the distance between $x$ and $y$ is at most 4. Then there exists a vertex $w$ in $x$-$y$ path, which is at a distance at most 2 from both $x$ and $y$. Consider $X=N(x)\cup N(y)$. Then $|N[X]\cap S|\leq 23$ and $|X|=8$. Then applying Theorem \ref{char ultra k secure}, we get $23\geq |N[X]\cap S|\geq 2|X|+\underset{z\in X}{\sum}|N[z]-S|=16+8=24$, a contradiction. Thus, the distance between any two vertices in $V-S$ is at least 5. 
	\end{proof}
	The next theorem gives bounds for $s_2^u(P_m\square P_n)$.
	\begin{theorem}\label{pmpn ultra 2 security number bound}
		For integers $m,n\geq 5$, $$ \left\lceil\frac{12mn+4m+4n-8\left\lfloor\frac{m}{5}\right\rfloor-8\left\lfloor\frac{n}{5}\right\rfloor-12}{13}\right\rceil \le s_2^u(P_m\square P_n)\le mn-\left\lfloor\frac{m}{5}\right\rfloor\left\lfloor\frac{n}{5}\right\rfloor.$$
	\end{theorem}
	\begin{proof} By Lemma \ref{ultra 2 secure pmpn lemma 2}, it directly follows that $s_2^u(P_m\square P_n)\leq mn-\left\lfloor\frac{m}{5}\right\rfloor\left\lfloor\frac{n}{5}\right\rfloor$. Let $S$ be an ultra $2$-secure set in $P_m\square P_n$. By Lemma \ref{ultra 2 secure pmpn lemma 1}, $\underset{v\in \partial S}{\cup}\big(N[N(v)]\cap S\big)\subseteq S$. By Lemma \ref{row-1-row-2}, $Row(i), Col(j)\subseteq S$, $i\in \{1,2,m-1,m\}$, $j\in \{1,2,n-1,n\}$. Thus, $$|S|\geq \big|\underset{v\in \partial S}{\cup}\big(N[N(v)]\cap S\big)\ \ \bigcup \underset{i \in \{1,2,m-1,m\}}{\cup}Row(i)\ \ \bigcup\underset{j \in \{1,2,n-1,n\}}{\cup}Col(j)\big|.$$
		
		\begin{itemize}
			\item By Lemma \ref{ultra 2 secure pmpn lemma 1} and Lemma \ref{ultra 2 secure pmpn lemma 2}, it follows that $V-S=\partial S$ and the members of $\{N[N(v)]\cap S: v\in \partial S\}$ are pairwise disjoint.
			\item By Lemma \ref{row-1-row-2}, $v=v_{ij}$ for some $i,j$, $2<i<m-1$, $2<j<n-1$. Hence $\underset{z\in N(v)}{\sum}|N[z]-S|=4$. Thus, by Theorem \ref{char ultra k secure}, for any $v\in \partial S$, $|N[N(v)]\cap S|\geq 2|N(v)|+\underset{z\in N(v)}{\sum}|N[z]-S|=8+4=12$.
			\item  By Lemma \ref{ultra 2 secure pmpn lemma 2},  $|Col(3)\cap (V-S)|\le \left\lfloor\frac{m}{5}\right\rfloor$. Therefore there can be at most $\left\lfloor\frac{m}{5}\right\rfloor$ vertices in $Col(1)$, which also belong to some member of the collection $\{N[N(v)]\cap S: v\in \partial S\}$ and there can be at most $3\left\lfloor\frac{m}{5}\right\rfloor$ such vertices in $Col(2)$. Similar argument holds for $Row(i), Col(j)$, $i\in \{1, 2, m-1, m\}$, $j\in \{n-1, n\}$.
		\end{itemize}
		Therefore, 
		\begin{eqnarray*}
			|S|&\geq&4m+4n-16+\underset{v\in V-S}{\sum}\big|N[N(v)]\cap S\big|-8\left\lfloor\frac{m}{5}\right\rfloor-8\left\lfloor\frac{n}{5}\right\rfloor+4\\&\geq& 4m+4n-8\left\lfloor\frac{m}{5}\right\rfloor-8\left\lfloor\frac{n}{5}\right\rfloor-12+12|V-S|
			\\&=&4m+4n-8\left\lfloor\frac{m}{5}\right\rfloor-8\left\lfloor\frac{n}{5}\right\rfloor-12+12(mn-|S|). 
		\end{eqnarray*}
		Thus we get $|S|\geq \frac{1}{13}\big(12mn+4m+4n-8\left\lfloor\frac{m}{5}\right\rfloor-8\left\lfloor\frac{n}{5}\right\rfloor-12\big)$, which completes the proof.
	\end{proof}
	\noindent The bound given in the above theorem is sharp. For $5\le m\le 7$, the upper bound matches the lower bound, and equality is achieved for all $n\ge 5$. Further, for all $m,n\leq 14$, the equality is attained at the lower bound specified in the above theorem. 
	
		By the above discussions, the problem of finding the ultra 2-security number of $P_m\square P_n$ is equivalent to find the maximum number of rhombuses of equal sides that can be packed inside a rectangle as shown in Figure \ref{ultra 2-security example}. 
	\begin{figure}[h]
		\tikzset{every picture/.style={line width=0.75pt}} 
		
		\begin{tikzpicture}[x=0.75pt,y=0.75pt,yscale=-1,xscale=1]
			
			\draw  [draw opacity=0] (108,12) -- (369,12) -- (369,273) -- (108,273) -- cycle ; \draw   (108,12) -- (108,273)(128,12) -- (128,273)(148,12) -- (148,273)(168,12) -- (168,273)(188,12) -- (188,273)(208,12) -- (208,273)(228,12) -- (228,273)(248,12) -- (248,273)(268,12) -- (268,273)(288,12) -- (288,273)(308,12) -- (308,273)(328,12) -- (328,273)(348,12) -- (348,273)(368,12) -- (368,273) ; \draw   (108,12) -- (369,12)(108,32) -- (369,32)(108,52) -- (369,52)(108,72) -- (369,72)(108,92) -- (369,92)(108,112) -- (369,112)(108,132) -- (369,132)(108,152) -- (369,152)(108,172) -- (369,172)(108,192) -- (369,192)(108,212) -- (369,212)(108,232) -- (369,232)(108,252) -- (369,252)(108,272) -- (369,272) ; \draw    ;
			\draw   (148,12) -- (188,52) -- (148,92) -- (108,52) -- cycle ;
			\draw   (188,72) -- (228,112) -- (188,152) -- (148,112) -- cycle ;
			\draw   (228,132) -- (268,172) -- (228,212) -- (188,172) -- cycle ;
			\draw   (268,192) -- (308,232) -- (268,272) -- (228,232) -- cycle ;
			\draw   (248,32) -- (288,72) -- (248,112) -- (208,72) -- cycle ;
			\draw   (288,92) -- (328,132) -- (288,172) -- (248,132) -- cycle ;
			\draw   (328,152) -- (368,192) -- (328,232) -- (288,192) -- cycle ;
			\draw   (148,152) -- (188,192) -- (148,232) -- (108,192) -- cycle ;
			\draw   (328,12) -- (368,52) -- (328,92) -- (288,52) -- cycle ;
			\draw  [fill={rgb, 255:red, 0; green, 0; blue, 0 }  ,fill opacity=1 ] (143.75,52) .. controls (143.75,49.65) and (145.65,47.75) .. (148,47.75) .. controls (150.35,47.75) and (152.25,49.65) .. (152.25,52) .. controls (152.25,54.35) and (150.35,56.25) .. (148,56.25) .. controls (145.65,56.25) and (143.75,54.35) .. (143.75,52) -- cycle ;
			\draw  [fill={rgb, 255:red, 0; green, 0; blue, 0 }  ,fill opacity=1 ] (243.75,72) .. controls (243.75,69.65) and (245.65,67.75) .. (248,67.75) .. controls (250.35,67.75) and (252.25,69.65) .. (252.25,72) .. controls (252.25,74.35) and (250.35,76.25) .. (248,76.25) .. controls (245.65,76.25) and (243.75,74.35) .. (243.75,72) -- cycle ;
			\draw  [fill={rgb, 255:red, 0; green, 0; blue, 0 }  ,fill opacity=1 ] (183.75,112) .. controls (183.75,109.65) and (185.65,107.75) .. (188,107.75) .. controls (190.35,107.75) and (192.25,109.65) .. (192.25,112) .. controls (192.25,114.35) and (190.35,116.25) .. (188,116.25) .. controls (185.65,116.25) and (183.75,114.35) .. (183.75,112) -- cycle ;
			\draw  [fill={rgb, 255:red, 0; green, 0; blue, 0 }  ,fill opacity=1 ] (223.75,172) .. controls (223.75,169.65) and (225.65,167.75) .. (228,167.75) .. controls (230.35,167.75) and (232.25,169.65) .. (232.25,172) .. controls (232.25,174.35) and (230.35,176.25) .. (228,176.25) .. controls (225.65,176.25) and (223.75,174.35) .. (223.75,172) -- cycle ;
			\draw  [fill={rgb, 255:red, 0; green, 0; blue, 0 }  ,fill opacity=1 ] (263.75,232) .. controls (263.75,229.65) and (265.65,227.75) .. (268,227.75) .. controls (270.35,227.75) and (272.25,229.65) .. (272.25,232) .. controls (272.25,234.35) and (270.35,236.25) .. (268,236.25) .. controls (265.65,236.25) and (263.75,234.35) .. (263.75,232) -- cycle ;
			\draw  [fill={rgb, 255:red, 0; green, 0; blue, 0 }  ,fill opacity=1 ] (323.75,192) .. controls (323.75,189.65) and (325.65,187.75) .. (328,187.75) .. controls (330.35,187.75) and (332.25,189.65) .. (332.25,192) .. controls (332.25,194.35) and (330.35,196.25) .. (328,196.25) .. controls (325.65,196.25) and (323.75,194.35) .. (323.75,192) -- cycle ;
			\draw  [fill={rgb, 255:red, 0; green, 0; blue, 0 }  ,fill opacity=1 ] (283.75,132) .. controls (283.75,129.65) and (285.65,127.75) .. (288,127.75) .. controls (290.35,127.75) and (292.25,129.65) .. (292.25,132) .. controls (292.25,134.35) and (290.35,136.25) .. (288,136.25) .. controls (285.65,136.25) and (283.75,134.35) .. (283.75,132) -- cycle ;
			\draw  [fill={rgb, 255:red, 0; green, 0; blue, 0 }  ,fill opacity=1 ] (143.75,192) .. controls (143.75,189.65) and (145.65,187.75) .. (148,187.75) .. controls (150.35,187.75) and (152.25,189.65) .. (152.25,192) .. controls (152.25,194.35) and (150.35,196.25) .. (148,196.25) .. controls (145.65,196.25) and (143.75,194.35) .. (143.75,192) -- cycle ;
			\draw  [fill={rgb, 255:red, 0; green, 0; blue, 0 }  ,fill opacity=1 ] (323.75,52) .. controls (323.75,49.65) and (325.65,47.75) .. (328,47.75) .. controls (330.35,47.75) and (332.25,49.65) .. (332.25,52) .. controls (332.25,54.35) and (330.35,56.25) .. (328,56.25) .. controls (325.65,56.25) and (323.75,54.35) .. (323.75,52) -- cycle ;
		\end{tikzpicture}\caption{A minimum ultra 2-secure set $S$ in $P_{14}\square P_{14}$. Black dots denote the vertices of $V-S$.}\label{ultra 2-security example}
	\end{figure}
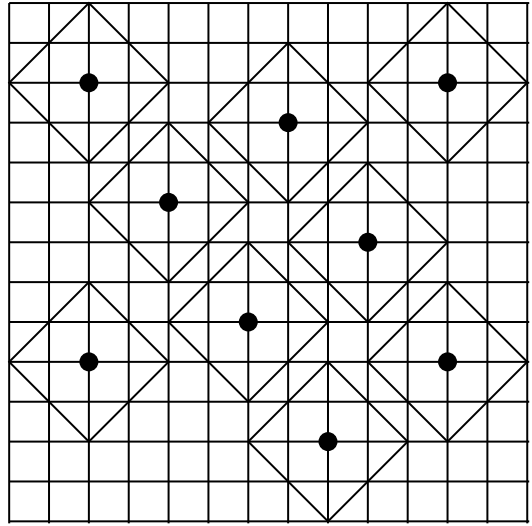

	Bounds for $s_2^u(P_m\square C_n)$ can be established that are similar to those in Theorem \ref{pmpn ultra 2 security number bound}. The following lemma is similar to Lemma \ref{row-1-row-2}.
	
	\begin{lemma}\label{ultra 2 secure pmcn lemma}
		Let $S$ be an ultra 2-secure set of $P_m\square C_n$. Then $Row(i) \subseteq S$ for all $i\in \{1, 2, m-1, m\}$. 
	\end{lemma}
	For integers $m,n$ with $\min\{m,n\}\leq 4$, $s_2^u(P_m\square C_n)=mn$. Note that Lemma \ref{ultra 2 secure pmpn lemma 1} and Lemma \ref{ultra 2 secure pmpn lemma 2} is true for $P_m\square C_n$ also. Therefore similar to Theorem \ref{pmpn ultra 2 security number bound}, the next  theorem can be proved using Lemma \ref{ultra 2 secure pmpn lemma 1}, Lemma \ref{ultra 2 secure pmpn lemma 2} and Lemma \ref{ultra 2 secure pmcn lemma}.
	
	\begin{theorem}
		For integers $m,n\geq 5$, $$ \left\lceil\frac{12mn+4n-8\left\lfloor\frac{n}{5}\right\rfloor}{13}\right\rceil \le s_2^u(P_m\square C_n)\le mn-\left\lfloor\frac{m}{5}\right\rfloor\left\lfloor\frac{n}{5}\right\rfloor.$$
	\end{theorem}

Roughly, $s_2(P_m\square C_n)=mn-x$, where $x$ is the maximum number of disjoint rhombuses that can be packed similarly as in Figure \ref{ultra 2-security example} on a surface of a cylindrical square grid. Also, $s_2(C_m \square C_n)=mn-y$ where $y$ is the maximum number of disjoint rhombuses that can be packed as in Figure \ref{ultra 2-security example} on a surface of a torus.
	
	\begin{theorem}
		For integers $m,n\geq 5$, $s_3^u(P_m\square P_n)=mn$.
	\end{theorem}
	\begin{proof}
		Let $S$ be an ultra 3-secure set in $P_m\square P_n$. Suppose $S$ is a proper subset of $V$. Then there exists a vertex $v\in \partial S$. Since every ultra 3-secure set is an ultra 2-secure set, by Lemma \ref{row-1-row-2}, $3\leq i\leq m-2$ and $3\leq j\leq n-2$. Further by Lemma \ref{ultra 2 secure pmpn lemma 1}, all the neighbors of $v$ are in $S$. Then for $X=N(v)\subseteq Bord(S)$, $12\geq |N[X]\cap S|\geq 3|X|+\underset{z\in X}{\sum}|N[z]-S|=12+4=16$, a contradiction.
		Therefore $S=V$ and hence $s_3^u(P_m\square P_n)=mn$. 
	\end{proof}
	
	Similarly, it can be proved that for $m,n\geq 3$, $s_2(P_2\square C_n)=2n$ and $s_3^u(P_m\square C_n)=s_3^u(C_m\square C_n)=mn$. 
	
	\section{Conclusion}
 In this paper, a valuable characterization characterization of ultra $k$-secure sets is derived. Some properties of $s^u_k(G)$ and its variation with the addition or deletion of edges are discussed. Further, the construction of graphs with a specified $s_k^u(G)$ value is demonstrated. Additionally, ultra $k$-security numbers of grid like graphs are determined for $k=1, 3$ and bounds are obtained for $k=2$.  
 The following are additional research problems to pursue along similar lines.
 \begin{enumerate}
 	\item Find the exact values of $s_2^u(G)$ for grid like graphs.
 	\item In a graph $G$, for any $k\in \{1, 2,\ldots,\Delta (G)-2\}$, $s_k^u(G)\leq s_{k+1}^u(G)$. What is the maximum possible value $l\in \{1, 2,\ldots,\Delta (G)-1\}$ such that the value of $s_k^u(G)$ strictly increases with $k$ up to $l$?
 	\item For a given $v\in V(G)$, find an algorithm to construct an ultra $k$-secure set containing $v$. Find the smallest cardinality of an ultra $k$-secure set containing $v$.
 	\item It is known that deciding whether a given set S is secure set in a graph is co-NP-complete. One may study the complexity of finding an ultra $k$-secure set.
 	 
 \end{enumerate} 
	
\section{Acknowledgment}
	The first author would like to thank University Grants Commission, New Delhi, India for Junior Research Fellowship (Ref. No. 191620207518 dated 20/07/2020). The authors would like to thank Prof. S T Hedetniemi, School of Computing, Clemson University, USA and all the reviewers  for valuable suggestions. The authors also  like to thank Dr. B Sooryanarayana, Dr. Ambedkar Institute of Technology, Bengaluru, India for constant guidance.\\
~~\\
\noindent
{\bf Funding:} Not applicable\\
{\bf Data Availability:}  Not applicable\\
{\bf Code  Availability:}  Not applicable\\

\section*{Declaration}
Authors declare that there are no conflicts of interest in conducting this research.

\end{document}